\documentclass[preprint,bj]{imsart}
\usepackage{dsfont}
\usepackage{bbm}
\usepackage[T1]{fontenc}
\usepackage{amsfonts,amssymb,amscd,amsthm,amsmath}
\usepackage{mathrsfs}
\usepackage[sort&compress, square, numbers]{natbib}
\usepackage[usenames]{color}
\RequirePackage[colorlinks,citecolor=blue,urlcolor=blue]{hyperref}
\newtheorem{thm}{Theorem}[section]
\newtheorem{cor}[thm]{Corollary}
\newtheorem{lem}[thm]{Lemma}
\newtheorem{prop}[thm]{Proposition}
\theoremstyle{definition}

\theoremstyle{remark}
\newtheorem{rem}[thm]{Remark}
\theoremstyle{Conjecture}
\newtheorem{conj}[thm]{Conjecture}
\numberwithin{equation}{section}

\startlocaldefs
\endlocaldefs

\newcommand{\ind}[1]{\mathbf{1}_{\{#1\}}}
\renewcommand{\P}{\mathbb{P}}
\newcommand{\E}{\mathbb{E}}
\newcommand{\R}{\mathbb{R}}
\newcommand{\T}{\mathbb{T}}
\newcommand{\N}{\mathbb{N}}

\newcommand{\I}{\mathscr{I}}
\newcommand{\F}{\mathscr{F}}
\newcommand{\D}{\mathscr{D}}
 \allowdisplaybreaks

\begin{document}

\begin{frontmatter}

\title{Second  and third orders asymptotic  expansions  for  the distribution of particles in  a  branching random walk with a random environment in time}

\runtitle{Asymptotic  Expansion  for a  Branching Random Walk}

\begin{aug}

\author{\fnms{ZhiQiang} \snm{Gao}\thanksref{a}\ead[label=e1]{gaozq@bnu.edu.cn}}
\and
\author{\fnms{Quansheng} \snm{Liu}\thanksref{b,c}\corref{}\ead[label=e2]{quansheng.liu@univ-ubs.fr}}
\address[a]{School of Mathematical Sciences, Laboratory of Mathematics and Complex Systems, Beijing Normal University, Beijing 100875, P. R. China, \printead{e1}}
\address[b]{Corresponding author, Universit\'e de  Bretagne-Sud, CNRS UMR 6205, LMBA, campus de Tohannic, F-{56000} Vannes, France, \printead{e2}}
\address[c]{Changsha University of Science and Technology, School of Mathematics and Statistics, Changsha {410004}, P. R. China }
\runauthor{Z.-Q Gao et Q. Liu}

\affiliation{Beijing Normal University and Universit\'e de Bretagne-Sud}

\end{aug}

\begin{abstract}
Consider  a branching random walk in which the offspring distribution and the moving law  both depend on an independent and identically  distributed random environment indexed by the time.
For the normalised  counting measure of  the number of particles of generation $n$ in a given region,  we give  the  second  and third orders  asymptotic  expansions of the central limit theorem under rather weak assumptions on the moments  of the underlying branching  and moving laws. The obtained results and the developed approaches  shed light on higher order expansions. In the proofs,  the Edgeworth expansion of  central limit theorems for sums of independent random variables,  truncating arguments and martingale approximation  play key  roles. In particular, we introduce a new  martingale, show its rate of convergence, as well as the rates of  convergence of some known martingales, which are of independent interest.
\end{abstract}




\end{frontmatter}

\section{Introduction}

A central limit theorem for the branching random walk  has been initiated and conjectured by Harris (1963, \cite[Chapter III. \S 16]{Harris63BP}). Since then this conjecture  has been  proved  in various forms and for various models, see e. g.   \cite{AsmussenKaplan76BRW1,AsmussenKaplan76BRW2,Klebaner82AAP,Biggins90SPA,GLW14,JoffeMoncayo73AM,Stam66,Yoshida08,Nakashima11}.
For the special cases where the underlying motion law is governed by the Wiener process or the  simple symmetric walk, R\'ev\'esz (1994, \cite{Revesz94}) investigated the speed of convergence in the central limit theorem and conjectured the exact  convergence rate,  which was confirmed by Chen (2001, \cite{Chen2001}) and complemented by Gao(2016, \cite{Gao2016}). Kabluchko(2012, \cite{Kabluchko12})
recovered and generalized Chen's results by using a general approach.  Gao and Liu (2016, \cite{GL14})  improved and extended  Chen's results on the branching Wiener process   to the  strongly non-lattice case under much weaker moment conditions. R\'ev\'esz, Rosen and Shi (2005, \cite{ReveszRosenShi2005}) found full  asymptotic expansions in  the local limit theorem for branching Wiener processes,   while  Gr\"{u}bel and Kabluchko (2015, \cite{GK15}) obtained the similar  result for a branching random walk on $ \mathbb{Z}$ and discussed the related applications in random trees. The  exact convergence rate obtained in   \cite{Chen2001,GL14} can be formulated  as the first order asymptotic expansion in the central limit theorem for the models considered therein.  Inspired by these works,   we consider  the following  natural question: what about the   asymptotic expansion of higher orders?

The  aim  of this article is to derive the second and third orders asymptotic expansions in the central limit theorem for a branching random walk with a time-dependent random environment. The goal is  twofold. On the one hand,  although central limit theorems for branching random walks have been well studied  and   the asymptotic expansions for branching Wiener processes  and  lattice branching random walks  were given in \cite{ReveszRosenShi2005} and \cite{GK15}, the asymptotic expansions  in central limit theorems for non-lattice branching random walks are still not known. On the other hand, we shall  perform our research in a more general framework, i.e. for \emph{a branching random walk with a random environment in time}, which is a natural generalization of classical branching random walk formulated in Harris \cite{Harris63BP}.
This model first appeared in Biggins (2004, \cite{BigginsKyprianou04AAP}) as a particular case  of a general framework,   and more related limit theorems  were given  in  \cite{Liu07ICCM,HuangLiangLiu14,WangHuang2016JTP}. For other different  kinds of  branching random walks in random environments, the reader may refer to \cite{BaillonClementGrevenHollander93,BovierHartung2015,GrevenHollander92PTRF,CometsPopov2007AOP,CometsPopov2007ALEA,
FangZeitouni2012,HuYoshida09,Nakashima11,Yoshida08,BirknerGK05,CometsYoshida2011JTP}. For other different aspects on  branching random walks, see \cite{Shi2015} and \cite{Zeitouni2012}.

This article opens the way to obtain   higher order asymptotic expansions.
 The second and third orders expansions given here  serve as good examples.
 The obtained results
 and the developed methods can be used to obtain asymptotic expansions
 of orders $4, 5$, etc., and hint  the general formula  for each finite order expansion, although we have not yet been able to prove it: see Conjecture \ref{conj-finite-order-exp} and the comments following it. We also mention that the approaches in our previous work \cite{GL14} have been significantly developed in the present  article.

The article is organized as follows. 
In Section \ref{EEB-sec2},  after giving  the rigorous definition of the model of a branching random walk with a random environment in time  and   introducing  three martingales,  we formulate the results on convergence rates of martingales as Theorem \ref{EEB-prop}, and then  state the main results on the asymptotic expansions in Theorems \ref{EEB-thm} and \ref{EEB-thm2}. Section  \ref{EEB-sec3} presents  some preliminaries including  a result on  the  Edgeworth expansion for the distribution function of sums of  independent random variables  and  a key decomposition used in the proofs.  Section \ref{EEB-sec4}  is devoted to  the proofs of main results. While the  proof of Theorem \ref{EEB-prop}  is postponed to Section \ref{EEB-sec5}.

\section{Main results}\label{EEB-sec2}

\subsection{Description of the model}
The model   \emph{a branching random walk with a random environment in time} can be formulated as follows \cite{GL14,GLW14}.
Let $(\Theta,\mathbbm{p})$ be a probability space, and $(\Theta^{\mathbb{N}},\mathbbm{p}^{\otimes \mathbb{N}})= (\Omega, \tau)$ be the corresponding product space.   For a sequence $\xi\in \Omega$, we denote $\xi=(\xi_1,\xi_2,\cdots)$, where $\xi_k$ are the $k-$th coordinate function on $\Omega$.  Then $\xi=(\xi_n)$ will serve as an independent and identically distributed environment. Let  $\theta$  be the usual shift transformation
on $\Theta^{\mathbb{N}}$: $\theta(\xi_0,\xi_1,\cdots)= (\xi_1,\xi_2,\cdots) $.
To each realization of $ \xi_n $ correspond  two probability distributions: the offspring distribution $p(\xi_n)
 = (p_0(\xi_n), p_1(\xi_n), \cdots ) $ on  $\N = \{0,1, \cdots\}$, and the moving distribution $ G (\xi_n) $ on $\R$.

Given the environment $ \xi=(\xi_n)$,  the  branching random walk in varying environment evolves according to the following rules:
 \begin{itemize}
   \item  At time $0$, an initial  particle $ \varnothing $ of generation $0$ is located at the origin $S_{\varnothing}=0$;
   \item  At time $1$,   $ \varnothing $ is replaced by $ N= N_{\varnothing}$  new particles of generation 1, and for $ 1\leq i \leq N$, each particle $\varnothing i$  moves to $S_{\varnothing i}=S_{\varnothing}+ L_{  i} $,  where
$N, L_{1}, L_2,  \cdots $ are mutually  independent, $N$ has the law $p(\xi_0)$,  and each $ L_i$ has the law $G(\xi_0)$.
   \item  At time $n+1$, each particle $u=u_1u_2\cdots u_n $ of generation $n$ is   replaced by   $ N_u$ new particles of generation $n+1$, with  displacements $ L_{u1}, L_{u2}, \cdots, L_{uN_u}$. That means  for $1\leq i\leq N_u$,  each particle $u i $ moves to $ S_{ui}= S_u+L_{ui} $, where
$N_u,  L_{u1}, L_{u2}, \cdots  $  are mutually  independent, $N_u$ has the law $ p(\xi_n)$, and each $L_{ui}$  has the same law $G(\xi_n)$. We do not assume the independence between $p(\xi_n)$ and $G(\xi_n)$, $n\geq 0$.
 \end{itemize}
By definition,
  given the environment $\xi$, the random variables $N_u$  and $L_u$, indexed by all the finite sequences $u$ of positive integers, are independent of each other.
For each realization $\xi \in \Theta^\N$ of the environment sequence,
let $(\Gamma, {  \mathcal{G}},  \mathbb{P}_\xi)$ be the probability space on which the
process is defined (when the environment $\xi$ is fixed to the given realization).  The probability
$\mathbb{P}_\xi$ is usually called \emph{quenched law}.
The total probability space can be formulated as the product space
$( \Theta^{\mathbb{N}}\times\Gamma , {  \mathcal{E}}^{\N} \otimes  \mathcal{G},   \mathbb{P})$,
 where $ \mathbb{P} = \E  (\delta_\xi \otimes \mathbb{P}_{\xi}) $ with $\delta_\xi $ the Dirac measure at $\xi$ and $\E$ the expectation with respect to the random variable $\xi$, so that  for all measurable and
 positive $g$ defined on $\Theta^{\mathbb{N}}\times\Gamma$, we have
  \[\int_{ \Theta^{\mathbb{N}}\times\Gamma } g (x,y) d\mathbb{P}(x,y) = \E  \int_\Gamma g(\xi,y) d\mathbb{P}_{\xi}(y).\]
The total
probability $\P$ is usually called \emph{annealed law}.
The quenched law $\P_\xi$ may be considered to be the conditional
probability of $\P$ given $\xi$. The expectation with respect to $\mathbb{P}$ will still be denoted by $\E$; there will be no confusion for reason of consistence.   The expectation with respect to
$\P_\xi$ will be denoted by $\E_\xi$.

Let $\mathbb{T}$ be the genealogical tree with $\{N_u\}$ as defining elements. By definition, we have:
(a) $\varnothing\in \mathbb{T}$; (b) $ui \in \mathbb{T}$ implies $u\in \mathbb{T}$; (c) if $ u\in \mathbb{T} $, then $ui\in \mathbb{T} $
if and only if $1\leq i\leq N_u $.
 Let $\mathbb{T}_n =\{u\in
\mathbb{T} :|u|=n\} $ be the set of particles of generation $n$, where $|u|$ denotes the length of the
sequence $u$ and represents the number of generation to which $u$ belongs.
\subsection{The main results}
Let  $Z_n(\cdot)$ be the counting measure of particles of generation $n$:
  for $B \subset \mathbb{R}$,
$$Z_n(B)= \sum_{u\in \mathbb{T}_n}  \mathbf{1}_{ B}(S_u).$$
 Then  $\{ Z_n(\mathbb{R})\}$ constitutes a branching process
in a random environment (see e.g. \cite{AthreyaKarlin71BPRE1,AthreyaKarlin71BPRE2,SmithWilkinson69}).
For $n\geq 0$, let $1_n = (1, \cdots, 1) $ be the sequence of $n$ times $1$,
with the  convention that $1_0= \emptyset$, and
set  $\widehat{N}_n = N_{1_n} $ (resp.  $\widehat{L}_n = L_{1_{n+1}}$), whose distribution under $\P_\xi$
  is the common one $p(\xi_n)$ (resp. $ G(\xi_n) $ ) of each $N_u$ (resp. $L_{ui}, i\geq 1$) with $|u|=n$,
  and  define
\begin{equation*}
   m_n= m(\xi_n)= \E_\xi \widehat{N}_{n}  ,\quad  \Pi_n = m_0\cdots m_{n-1}, \quad   \Pi_0=1.
\end{equation*}
It is well known that
  the normalized sequence  $$W_n=\Pi_n^{-1} Z_n(\mathbb{R}), \quad n\geq 1$$
    constitutes a martingale with respect to the filtration $\mathscr{F}_n$ defined by:  $$ \mathscr{F}_0=\{\emptyset,\Omega \} ,   \mathscr{F}_n =\sigma ( \xi, N_u:|u| < n), \mbox{ for  }n\geq 1. $$
    Throughout the article, we  shall always assume the following conditions:
 \begin{equation}\label{EEB-eq2.1}
    \E \ln m_0>0   \quad { \mbox{and}}\quad  {\E}\left(\frac{1}{m_0} \widehat{N}_{0} \left(\ln^+ \widehat{N}_{0}\right)^{1+\lambda} \right)<\infty   ,
    \end{equation}
where the value of $\lambda >0$ will be specified in the hypothesis of theorems,  and $\ln^+ x= \max(\ln x, 0)$ ( resp.  $ \ln^- x = \max( - \ln x, 0) $ ) denotes the positive (resp. negative) part of $\ln x$ for  $x>0$.
     It is well known that the  limit
    \begin{equation*}
    W=\lim_n W_n
\end{equation*}
exists almost surely (a.s.) by the martingale convergence theorem, and that,   under \eqref{EEB-eq2.1}, $\E W =1$ and   $W>0$   a.s.  on the explosion event $\{Z_n(\R) \rightarrow \infty\}$ (in fact \eqref{EEB-eq2.1} with $\lambda =0$ suffices for these assertions:  see  \cite{AthreyaKarlin71BPRE2} and \cite{Tanny1988SPA}).
In particular, the underlying  branching process  is  \emph{supercritical} and
$Z_n(\R) \rightarrow \infty$ with positive probability.

For $n\geq 0$, define
\begin{eqnarray*}
   && l_n   = \E_\xi \widehat{L}_n, \quad
    \sigma_n^{(\nu)} =\E_\xi \big(\,\widehat{L}_n- l_n\big)^\nu \; \mbox{ for } \nu\geq 2;\\
    &&
    \ell_n= \sum_{k=0}^{n-1} l_{k},  \quad s_n^{(\nu)} = \sum_{k=0}^{n-1} \sigma_k^{(\nu)} \; \mbox{ for } \nu\geq 2, \quad
 s_n  = \big( s_n^{(2)}\big)^{1/2}.
\end{eqnarray*}
Since $\{\xi_n\} $ are  i.i.d, by the law of large numbers,  we see that
\begin{equation*}
 s_n^{(\nu)} \sim   {n  \E \sigma_0^{(\nu)}},
\end{equation*}
where $a_n\sim b_n $ means $\lim_{n\rightarrow \infty} a_n/b_n=1$. This will be frequently used later.

%


To state our main result, we shall need the following  martingales:
\begin{eqnarray*}
    & &N_{1,n} = \frac{1}{\Pi_n} \sum_{u\in \T_n} (S_u-\ell_n),   \\
    & & N_{2,n} = \frac{1}{\Pi_n} \sum_{u\in \T_n} \Big[(S_u-\ell_n)^2 -s_n^2 \Big], \\
    && N_{3,n} = \frac{1}{\Pi_n} \sum_{u\in \T_n} \Big[(S_u-\ell_n)^3 -3(S_u-\ell_n)  s_n^2-  s_n^{(3)}   \Big],
\end{eqnarray*}
  with respect to
 the filtration $ (\mathscr{D}_n ) $  defined by $$ \mathscr{D}_0=\{\emptyset,\Omega \}, \quad  \mathscr{D}_n  =  \sigma ( \xi, N_u, L_{ui}:  i\geq 1, |u| <
  n)  \mbox{ for  $n\geq 1$}.$$
\begin{thm}[Convergence rates of the martingales] \label{EEB-prop}
 The sequences $ \{ (N_{\nu,n}, \mathscr{D}_n)\} (\nu=1,2,3)$  are martingales. Moreover, we have the following assertions about their rate of convergence:
   \begin{description}
    \item[(1)]  Assume \eqref{EEB-eq2.1} and $\E (\ln^- m_0)^{1+\lambda}<\infty$   for some $ \lambda>1$,    together with  $     \E\big(| \widehat{L}_0 |^{\eta } \big) <\infty $ for some  $\eta>2$. Then  there exists a real random variable $V_1$  such that a.s. $$ \qquad N_{1,n}-V_1=o(n^{-\lambda+1+\delta} )  \quad  \forall \delta>0. $$
    \item[(2)]  Assume \eqref{EEB-eq2.1}  and $\E (\ln^- m_0)^{1+\lambda}<\infty$ for some $ \lambda>2$,       together with  $     \E\big(| \widehat{L}_0 |^{\eta } \big) <\infty $ for some $\eta>4$. Then  there exists a real random variable $V_2$  such that  a.s. $$ \qquad N_{2,n}-V_2=o(n^{-\lambda+2+\delta} )  \quad  \forall \delta>0. $$
    \item[(3)]  Assume \eqref{EEB-eq2.1} and $\E (\ln^- m_0)^{1+\lambda}<\infty$  for some $ \lambda>3$,  together with $     \E\big(| \widehat{L}_0 |^{\eta } \big) <\infty $ for some $\eta>6$. Then  there exists a real random variable $V_3$  such that  a.s. $$ \qquad N_{3,n}-V_3=o(n^{-\lambda+3+\delta} )  \quad  \forall \delta>0. $$
  \end{description}
 \end{thm}
 The proof is postponed to Section \ref{EEB-sec5}.

\begin{rem} \label{rem-EEB-prop}
  A weaker version of parts (1) and (2)  has been proved in  \cite[Propositions 2.1 and 2.2]{GL14}, where the convergence of the martingales $(N_{1,n})$  and $(N_{2,n})$ was shown  under the same conditions.
  The martingale
$(N_{3,n})$ appears for the first time in this article.
\end{rem}


For asymptotic expansions of the central limit theorem, we  will need   the following hypotheses on the motion law $G(\xi_0)$ of   particles:
\begin{equation}\label{EEB-eq2.3}
      \P \Big( \limsup_{|t|\rightarrow \infty }\big|\E_\xi e^{it\widehat{L}_0}\big| <1 \Big) >0   \quad \mbox{  and } \quad   \E\big(| \widehat{L}_0 |^{\eta } \big) <\infty  , \end{equation}
where the value of $\eta>1 $ will  be specified  in  the theorems.  The first hypothesis  means that Cram\'er's condition about the  characteristic function of $G(\xi_0)$  holds with positive probability.
Set  $$ Z_n(t)=Z_n((-\infty, t]), \quad  \phi(t)=\frac{1}{\sqrt{2\pi}}e^{-t^2/2}, \quad \Phi(t) = \int_{-\infty}^{t}\phi(x) \mathrm{d}x,  \quad t\in \R .$$
Denote by $H_m(\cdot)$   the Chebyshev-Hermite polynomial of degree $m$:
\begin{equation*}
   H_{m}(x)=m! \sum_{k=0}^{\lfloor \frac{m}{2}\rfloor} \frac{(-1)^k x^{m-2k}}{ k!(m-2k)! 2^k},
\end{equation*}
where $\lfloor x\rfloor$  denotes the largest integer not bigger than $x$. More precisely, we need the following polynomials:
\begin{eqnarray*}
   & &  H_0(x)=1, \quad  H_1(x) =x,  \quad   H_2(x)=x^2-1,   \quad  H_3(x)= x^3-3x,  \\
    & & H_4(x)= x^4-6x^2+3,    \quad H_5 (x)= x^5-10x^3+15x.
\\
  && H_6(x)= x^6-15x^4+45x^2-15,
  \quad   H_8(x)= x^8-28x^6+210x^4-420x^2+105.
\end{eqnarray*}

In \cite[Theorem 2.3]{GL14}, the authors proved the following result about the exact rate of convergence in the central limit theorem:
   if  $\E m_0^{-\delta}<\infty$ for some $\delta>0$, \eqref{EEB-eq2.1} holds  for some $\lambda>8$ and  \eqref{EEB-eq2.3}  holds  for some $ \eta>12$,     then  for all $ t\in \R$,
 \begin{equation*}
   \sqrt{n}\Big[\frac{1}{\Pi_n} Z_n(\ell_n+s_n t)  - \Phi(t) W\Big] \xrightarrow{n \rightarrow \infty }  - \frac{ \phi(t)\; V_1 }{  (\E \sigma_0^{(2)})^{1/2} }   -\frac{(\E \sigma_0^{(3)}) \,  H_2(t)\; \phi(t) \; W }{6 \big(\E \sigma_0^{(2)}\big)^{3/2} } \quad \mbox {a.s.},
 \end{equation*}
From this result  we  can deduce  the following version describing the first order expansion in the central limit theorem: for $t\in \R$, { as $n \rightarrow \infty$, }
 \begin{equation}\label{spa-cbrweq4}
    \frac{1}{\Pi_n} Z_n(\ell_n+s_n t)=\bigg(\Phi(t) +\frac{  Q_{1,n}(t)}{n^{1/2}} \bigg)W +\Big(-\frac{1}{s_n}\Big)\phi(t)V_1 + o\Big(\frac{1}{\sqrt{n}}\Big)   \mbox{ a.s.},
 \end{equation}
 where
 \begin{equation}\label{EEB-eq2.6}
   \frac{  Q_{1,n}(t)}{n^{1/2}}  = -  \frac{s_n^{(3)}}{6s_n^3 } H_2(t) \phi(t).
 \end{equation}
In this article, we are interested in higher order expansions.
Our main results are the following two theorems about the second and third orders expansions in the central limit theorem. Naturally, for a higher order expansion, we need higher order moment conditions.
\begin{thm}[Second order expansion] \label{EEB-thm}
Assume  $\E m_0^{-\delta}<\infty$ for some $\delta>0$,  \eqref{EEB-eq2.1} for some $\lambda>18$ and \eqref{EEB-eq2.3} for some $ \eta>24$.   Then
  for $t\in \R$, { as $n \rightarrow \infty$, }
  \begin{multline}\label{EEB-eq2.5}
   \frac{1}{\Pi_n}  Z_n( \ell_n + s_n t ) =\Bigg (\Phi(t)+\frac{  Q_{1,n}(t)}{n^{1/2}} +\frac{  Q_{2,n}(t)}{n} \Bigg) W\\ +\Big(-\frac{1}{s_n}\Big)\Bigg( \phi(t)+ \frac{Q'_{1,n}(t)}{n^{1/2}}\Bigg)V_1  +\frac{1}{2!}\frac{1}{s_n^2} \phi'(t)V_2   +o\big(\frac{1}{n}\big)  \mbox{~~ a.s. },
  \end{multline}
 \mbox{ where } $Q_{1,n}$ is defined   by \eqref{EEB-eq2.6}  and
 \begin{eqnarray}
 \label{EEB-eq2.7}
    \frac{  Q_{2,n}(t)}{n}  &=&  -\frac{(s_n^{(3)})^{2}}{72s_n^{6}}H_5(t)\phi(t)-\frac{1}{24s_n^{4}}\sum_{j=0}^{n-1}\Big (\sigma_j^{(4)} -3 \big(\sigma_j ^{(2)}\big) ^2\Big)H_3(t)\phi(t).
 \end{eqnarray}
\end{thm}

\begin{thm}[Third order expansion] \label{EEB-thm2}
Assume  $\E m_0^{-\delta}<\infty$ for some $\delta>0$,    \eqref{EEB-eq2.1} for some  $\lambda>32$  and \eqref{EEB-eq2.3} for some $ \eta>40$.   Then
  for $t\in \R$, { as $n \rightarrow \infty$, }
  \begin{multline}\label{EEB-eq2.8}
   \frac{1}{\Pi_n}  Z_n( \ell_n + s_n t ) =\Bigg (\Phi(t)+\sum_{\nu=1}^3\frac{  Q_{\nu,n}(t)}{n^{\nu/2}}  \Bigg) W +\Big(-\frac{1}{s_n}\Big)\Bigg( \phi(t)+\sum_{\nu=1}^2\frac{  Q'_{\nu,n}(t)}{n^{\nu/2}} \Bigg)V_1 \\ +\frac{1}{2!}\frac{1}{s_n^2}\Bigg (\phi'(t) + \frac{Q''_{1,n}(t)}{n^{1/2}}\Bigg )V_2   +   \frac{1}{3!}\frac{-1}{s_n^3} \phi''(t)    V_3+o\big(\frac{1}{n^{3/2}}\big)  \mbox{~~ a.s. },
  \end{multline}
 \mbox{ where } $Q_{1,n}$, $Q_{2,n}$  are defined   by \eqref{EEB-eq2.6}   and \eqref{EEB-eq2.7},  and
 \begin{eqnarray}
  \nonumber     \frac{  Q_{3,n}(t)}{n^{3/2}}  &=&  -\frac{\big(s_n^{(3)} \big)^3}{1296 s_n^{9}} H_8(t) \phi(t) -\frac{1}{120 s_n^{5}}\sum_{j=0}^{n-1} \Big(\sigma_j^{(5)} -10\sigma_j^{(3)}\sigma_j^{(2)} \Big)H_4(t)  \phi(t)\\ && -\frac{ s_n^{(3)} } { 144  s_n^{7}} \sum_{j=0}^{n-1} \Big(\sigma_j^{(4)} -3\big(\sigma_j^{(2)}\big)^2  \Big)H_6(t)  \phi(t)  .
 \end{eqnarray}
\end{thm}

The reduced versions  of \eqref{spa-cbrweq4} and  \eqref{EEB-eq2.5}  for a branching random walk with a deterministic environment  have been announced in \cite{GL16CRAS}.

Notice that when the branching random walk  dies out, then  $Z_n(\mathbb{R})=0 $ for $n$ large enough, so that  $W_n=N_{1,n}=N_{2,n}=N_{3,n}=0$, hence the expansions  \eqref{spa-cbrweq4}, \eqref{EEB-eq2.5}  and \eqref{EEB-eq2.8} becomes trivial.

From Theorems \ref{EEB-thm}  and  \ref{EEB-thm2},  we can derive  the second and third  order expansions for the branching Wiener process,  where the underlying branching process is a Galton-Watson process whose  offspring distribution has mean $m>1$ and the motion of particles is governed by  a Wiener process.
For example, applying Theorem  \ref{EEB-thm2} to a constant environment and to a Gaussian moving law (for which   the condition   \eqref{EEB-eq2.3}  is  valid for all $\eta>0$), we obtain:
\begin{cor} [Third  order expansion for the branching Wiener process]  \label{EEBcor}
For the branching Wiener process whose offspring distribution $\{p_k: k\geq 0\}$  satisfies
 $m:= \sum_{k=0}^\infty k p_k >1$ and $\sum_{k=0}^\infty  k (\ln k)^{1+\lambda} p_k <\infty$  for some $\lambda>32$, we have,
for $t\in \R$, { as $n \rightarrow \infty$},
\begin{equation*}
   \frac{1}{m^n}{Z_n(  \sqrt{n} t ) } =  \Phi(t) W- \frac{1}{\sqrt{n} } \phi(t) V_1- \frac{1}{2 n}t\phi(t)V_2- \frac{1}{6 n^{3/2}}(t^2-1)\phi(t)V_3+  o\big(\frac{1}{n^{3/2}}\big)  \mbox{~~ a.s. }
\end{equation*}
\end{cor}
\begin{rem}
 (1) This corollary   extends    \cite[Theorem 3.2]{Chen2001}, which  gave   the first  order expansion of the central limit theorem under the second moment condition $\sum_k k^2 p_k <\infty$ for this model.  It should  be mentioned  that   in  \cite{ReveszRosenShi2005}, the full expansion for the local limit theorem was obtained for the same model.  However,     Corollary  \ref{EEBcor}  cannot be  derived   from the expansion in \cite{ReveszRosenShi2005}  (and vice versa). \; (2) A similar result can be easily formulated for the branching Wiener process in a random environment.
\end{rem}

Inspired by Theorems \ref{EEB-thm} and \ref{EEB-thm2}, we have the following conjecture for the asymptotic expansion of
 finite  order.

\vskip 3mm

 \begin{conj}[Finite order expansion] \label{conj-finite-order-exp}
  Let  $\kappa\geq 1$ be an integer.
  Assume  $\E m_0^{-\delta}<\infty$ for some $\delta>0$,    \eqref{EEB-eq2.1}  and \eqref{EEB-eq2.3} for some $\lambda >0$  and $ \eta >0 $ large enough.    Then
    \begin{multline}\label{EEB-eq2.9}
   \frac{1}{\Pi_n}  Z_n( \ell_n + s_n t ) =\Bigg (\Phi(t)+\sum_{\nu=1}^\kappa\frac{  Q_{\nu,n}(t)}{n^{\nu/2}}  \Bigg) W\\ +\sum_{j=1}^\kappa \frac{1}{j!} (-\frac{1}{s_n})^{j}   V_j\frac{d ^j}{ dt^j}  \Bigg (\Phi(t)+\sum_{\nu=1}^{\kappa-j}\frac{  Q_{\nu,n}(t)}{n^{\nu/2}}  \Bigg)   +o\big(\frac{1}{n^{\kappa/2}}\big)  \mbox{~~ a.s. },
  \end{multline}
where  $V_j$ are real random variables,  and
\begin{align*}
Q_{\nu,n}(x)&=  \sum { }^{'}(-1)^{\nu+2s}\Phi^{(\nu+2s)}(x) \prod_{m=1}^{\nu} \frac{1}{k_m!} \bigg(\frac{\lambda_{m+2,n}}{(m+2)!}\bigg)^{k_m}
    \\  & =- \phi(x)\sum { }^{'}  H_{\nu+2s-1}(x)\prod_{m=1}^{\nu} \frac{1}{k_m!} \bigg(\frac{\lambda_{m+2,n}}{(m+2)!}\bigg)^{k_m},
\end{align*}
  with   the summation  $ \sum { }^{'}$  being carried out over all  nonnegative integer solutions $(k_1, \dots, k_\nu  )$  of the equation   $k_1+2k_2+\cdots +\nu k_{\nu}=\nu$,    $s= k_1+\cdots+  k_\nu  $  and
\begin{align*}
      \quad   &   \lambda_{\nu,n }= n^{(\nu-2)/2} s_n^{-2\nu}  \sum_{j=0}^{n-1} \gamma_{\nu j},  \quad {\nu=3,4\cdots, k};\\
     &  \gamma_{\nu j}= \frac{1}{i^{\nu}}\left[ \frac{d^{\nu}}{ dt^{\nu}} \log \E_\xi e^{it  (\widehat{L}_j- l_j )} \right]_{t=0}, \qquad {\nu =1,2,\cdots }.
\end{align*}
\end{conj}

We remind that the term $ \Phi(t)+\sum_{\nu=1}^{\kappa  }  {n^{-\nu/2}}{  Q_{\nu,n}(t)}$
is the Edgeworth expansion of the distribution function of   sums of the random variables
$\widehat{L}_0, \widehat{L}_1,\cdots$.  See Lemma \ref{lem-Edge-exp} below.   The reader may refer to \cite{Petrov75} for more information  on the Edgeworth expansion.

For $\kappa =1, 2, 3$, the expansion is given respectively by  \eqref{spa-cbrweq4}, Theorem
\ref{EEB-thm} and Theorem \ref{EEB-thm2}. By using the method proposed in this article, we should be  able to  prove,  through tedious analysis,  the expansion formula for order $\kappa =4, 5, $ etc.  However,
     we have not yet found a simple and unified method valid for all $\kappa \geq 1$.    This seems to  need a great deal of  effort and will be our future aim.

For the proofs of Theorems \ref{EEB-thm} and \ref{EEB-thm2}, we further develop   the approaches  used in \cite{GL14}. Like in \cite{GL14}, the basic ideas are
the Edgeworth expansion for an approximation of the cumulative
distribution function of the sum of independent random variables (to control the position of  particles in $n$-th generation, which  makes appear the Chebyshev-Hermite polynomials), the approximation by martingales, and a suitable decomposition of $Z_n(A)$ involving the independence of each particle (conditionally on the
environment) from time $ k_n = \lfloor n^\beta \rfloor $ for some $\beta \in (0,1)$  (see \eqref{EEBeq3-1}), where $\lfloor n^\beta \rfloor $ denotes the integral part of $n^\beta$.
However, the adaption of the approaches in \cite{GL14} (proposed  for the first order)
to higher orders is far from being evident, and the progress of the approaches in the present article is significant.  Actually, to obtain the higher order expansions, we perform much more effort
 than in \cite{GL14}. This can be seen through three aspects.   Firstly, we need to extract more  terms from the Edgeworth expansion by using Taylor's expansion, which are rather tedious due to the complexity of the Edgeworth expansion.  Secondly, we should carefully  analyze the extracted terms  and suitably combine them; in particular   we  need to find out new martingales
which appear in describing the higher order expansion, and show their convergence and  their rate of convergence;  furthermore, even for the known martingales  $(N_{1,n})$ and $(N_{2,n})$, we need to investigate their convergence rates, which were not studied in the  previous work \cite{GL14}.
Thirdly,  the
time $k_n$   for the decomposition of the branching random walk needs to be in a smaller time range (than in \cite{GL14}),  to guarantee the Edgeworth expansion at a next order
during the remaining time interval.

\medskip
For simplicity and without any loss of generality, hereafter we will always assume that $l_n=0$ (otherwise, we only need to replace  $L_{ui}$ by $L_{ui}-l_n$) and hence $\ell_n=0$. In the  following, we will use $K_\xi $ as a constant depending on the environment, which may change from line to line.
\section{Preliminary results}\label{EEB-sec3}

\subsection{The Edgeworth expansion for sums of independent random variables}
To begin with, we present the Edgeworth expansion  for the distribution function of sums of independent random variables, which  is needed  to prove the main theorems.    We recall the version obtained by Bai and Zhao (1986, \cite{BaiZhao1986}),  which generalizes the i.i.d. case \cite[p.159, Theorem 1]{Petrov75}.

Let $\{X_j\}$ be  independent random variables,  satisfying for each $j\geq 1$
\begin{equation}\label{cbrwa1}
   \E X_j=0 \;\; \mbox{ and } \;\;  \E |X_j|^{k} <\infty
\end{equation}
 for some integer   $k \geq 3$. We write  $B_n^2 = \sum_{j=1}^{n} \E X_j^2$ and only consider the nontrivial case $B_n>0$.
Let $\gamma_{\nu j}$  be  the $\nu$-order cumulant of $X_j$  for each $j\geq1$, defined by
\begin{equation*}
   \gamma_{\nu j}= \frac{1}{i^{\nu}}\left[ \frac{d^{\nu}}{ dt^{\nu}} \log \E e^{it X_j } \right]_{t=0}, \qquad {\nu =1,2,\cdots }.
\end{equation*}
Write
\begin{align*}
    &  \lambda_{\nu,n }= n^{(\nu-2)/2} B_n^{-\nu}  \sum_{j=1}^n \gamma_{\nu j},  \quad {\nu=3,4\cdots, k}; \\
    &  Q_{\nu,n}(x)=  \sum { }^{'}(-1)^{\nu+2s}\Phi^{(\nu+2s)}(x) \prod_{m=1}^{\nu} \frac{1}{k_m!} \bigg(\frac{\lambda_{m+2,n}}{(m+2)!}\bigg)^{k_m}
    \\ & \qquad \quad =- \phi(x)\sum { }^{'}  H_{\nu+2s-1}(x)\prod_{m=1}^{\nu} \frac{1}{k_m!} \bigg(\frac{\lambda_{m+2,n}}{(m+2)!}\bigg)^{k_m},
\end{align*}
where the summation  $ \sum { }^{'}$  is carried out over all  nonnegative integer solutions $(k_1, \dots, k_\nu  )$  of the equation $k_1+2k_2+\cdots +\nu k_{\nu}=\nu $  and $s=k_1+\cdots+  k_\nu$.

 For   $ 1\leq j\leq n$   and   $x\in \R$, define
\begin{align*}
    & F_n(x)= \P \Big (  {B_n}^{-1}  \sum_{j=1}^n X_j  \leq x\Big),    \quad v_j(t) = \E e^{itX_j}; \\
    &Y_{nj}= X_j \mathbf{1}_{\{ |X_j| \leq B_n\}},  \quad    Z_{nj}^{(x)}= X_{j }\mathbf{1}_{ \{|X_j| \leq B_n(1+|x|)\}}, \quad W_{nj}^{(x)}= X_{j }\mathbf{1}_{ \{|X_j| > B_n(1+|x|)\}}.
\end{align*}
The  Edgeworth expansion  theorem can be stated as follows.
\begin{lem} [\cite{BaiZhao1986}] \label{lem-Edge-exp}
  Let $n\geq 1$ and $X_1, \cdots, X_n$ be a sequence of  independent random variables satisfying    $ B_n>0$.   Let $k\geq 3$ be an integer such that \eqref{cbrwa1}  holds. Then
\begin{multline*}
   | F_n(x) - \Phi(x)- \sum_{\nu=1}^{k-2} Q_{\nu n}(x)n^{-\nu/2} | \leq   C(k)\Bigg\{   (1+|x|)^{ -k} B_n^{-k} \sum_{j=1}^n  \E |W_{nj}^{(x)}|^k +    \\   (1+|x|)^{ -k-1} B_n^{-k-1} \sum_{j=1}^n\E |Z_{nj}^{(x)}|^{k+1}  +   (1+|x|)^{ -k-1} n^{k(k+1)/2}\Big( \sup_{|t|\geq \delta_n} \frac{1}{n} \sum_{j=1}^n |v_{j}(t)| +\frac{1}{2n} \Big)^n \Bigg\},
   \end{multline*}
where $\displaystyle \delta_n =  \frac{1}{12}  {B_n^2}{  (\sum_{j=1}^n\E  |Y_{nj}|^3)^{-1}  }$,  $C(k)>0 $  is a constant depending only on $k$.
 \end{lem}

\subsection{Notation and a key decomposition} We first introduce some notation which will be used in the sequel.

In addition to the $\sigma-$fields $\mathscr{F}_n$ and  $\mathscr{D}_n$,  the following  $\sigma$-fields  will also be used:
\begin{eqnarray*}
  \mathscr{I}_0=\{\emptyset,\Omega \}, \quad  \mathscr{I}_n &=&   \sigma ( \xi_k, N_u, L_{ui}: k<n, i\geq 1, |u| <
  n) \mbox{ for  $n\geq 1$}.
\end{eqnarray*}
Define the following conditional probabilities and  expectations:
\begin{eqnarray*}
 &&\P_{\xi, n}(\cdot ) = \P_\xi(\cdot | \D_n), \quad   \E_{\xi,n}(\cdot )= \E_\xi(\cdot | \D_n);\quad \P_{n}(\cdot )= \P(\cdot | \I_n), \quad  \E_{n}(\cdot )= \E(\cdot | \I_n).
\end{eqnarray*}

As usual, we write $\N^* = \{1,2,3,\cdots \}$ and denote by
$$ U= \bigcup_{n=0}^{\infty} (\N^*)^n $$
the set of all finite sequences, where $(\N^*)^0=\{\varnothing \}$ contains the null sequence $ \varnothing$.

For all $u\in U$, let $\mathbb{T}(u)$ be the shifted tree of $\mathbb{T}$ at $u$  with defining elements $\{N_{uv}\}$: we have
1) $\varnothing \in \mathbb{T}(u)$, 2) $vi\in \mathbb{T}(u)\Rightarrow v\in \mathbb{T}(u)$ and  3) if  $v\in \mathbb{T}(u)$, then $vi\in \mathbb{T}(u)$ if and only if $1\leq i\leq N_{uv} $. Define $\mathbb{T}_n(u)=\{v\in \mathbb{T}(u): |v|=n\}$. Then  $\mathbb{T}=\mathbb{T}(\varnothing)$ and $\mathbb{T}_n=\mathbb{T}_n(\varnothing)$.

 For $u\in (\N^*)^k (k\geq 0)
 $  and $n\geq 1$,
let $S_{u}$ be the position of $u$ and write
\begin{equation*}
Z_{n}(u,B)= \sum_{v \in \mathbb{T}_n(u)}\mathbf{1}_B(S_{uv}-S_u), \quad Z_{n}(u,t)= Z_{n}\Big(u, (-\infty,t]\Big).
\end{equation*}
Then  the law of $Z_{n}(u,B)$ under ${\P}_{\xi}$ is the same as that of $Z_n(B)$ under ${P}_{\theta^k\xi}$. Define
   \begin{eqnarray*}
  &&W_{n}(u,B) =Z_{n}(u,B)/\Pi_n(\theta^k\xi), \quad  W_n(u,t) = W_n(u,(-\infty,t]), \\
  && W_{n}(B) =Z_{n}( B)/ \Pi_n, \quad W_n(t) =W_n((-\infty,t]).
\end{eqnarray*}
By definition, we have $\Pi_{n}(\theta^k\xi)=m_k\cdots m_{k+n-1}$, $Z_n(B)= Z_n(\varnothing, B)$, $W_n(B)=W_n(\varnothing,B)$, $W_n= W_n(\mathbb{R})$.

For each $n$,
we choose an integer $k_n<n$ as follows. Let $\beta \in (0,1)$ whose value will be suitably fixed in the proofs and  set $k_n=\lfloor n^{\beta}\rfloor$,  the greatest integer not bigger than $ n^{\beta}$.
It is apparent   that
$$Z_n(s_nt)= \sum_{u\in \T_{k_n} }  Z_{n-k_n} ( u, s_nt-S_u), $$
from which we have the following important decomposition:
\begin{equation}\label{EEBeq3-1}
  \frac{1}{\Pi_n }Z_n(   s_n t )= \mathbb{A}_n + \mathbb{B}_n,
\end{equation}
 with
\begin{align*}
 &   \mathbb{A}_n= \frac{1}{\Pi_{k_n }}   \sum_{u\in \mathbb{T}_{k_n}}
  \left[W_{n-k_n}(u,s_n t-S_u)-  {\E}_{\xi,k_n} W_{n-k_n}(u,s_n t-S_u)\right],\\  \nonumber
  & \mathbb{B}_n= \frac{1}{\Pi_{k_n }}   \sum_{u\in \mathbb{T}_{k_n}}{\E}_{\xi,k_n} W_{n-k_n}(u,s_n t-S_u).
\end{align*}

\section{Proofs of   Theorems \ref{EEB-thm} and \ref{EEB-thm2} }\label{EEB-sec4}
\subsection{Outline of proofs}
In our proofs, we shall need  the following truncations of the martingales (recall that we   assume  $\ell_n=0$):
 \begin{eqnarray}
    \label{EEBeq3-5}      & & \overline{W}_{k_n}=\frac{1}{\Pi_{k_n}} \sum_{u\in \T_{k_n}}  \mathbf{1}_{\{|S_u| \leq k_n\}} ; \quad  \overline{N}_{1,k_n}= \frac{1}{\Pi_{k_n}} \sum_{u\in \T_{k_n}} S_u \mathbf{1}_{\{|S_u| \leq k_n\}}; \\
\label{EEBeq3-6}          & &\overline{N}_{2,k_n}= \frac{1}{\Pi_{k_n}} \sum_{u\in \T_{k_n}}( S^2_u-s_n^2) \mathbf{1}_{\{|S_u| \leq k_n\}}; \\
 \label{EEBeq3-7}         & & \overline{N}_{3,k_n}= \frac{1}{\Pi_{k_n}} \sum_{u\in \T_{k_n}}( S_u^3-  3S_us_n^2-s_n^{(3)}) \mathbf{1}_{\{|S_u| \leq k_n\}}.
      \end{eqnarray}

Notice that   the condition $ \E m_0^{-\delta}<\infty$ for some $\delta>0$ implies that $\E \big(\ln^- m_0\big)^{ \kappa }<\infty$  for all $\kappa>0$. Therefore Theorem \ref{EEB-prop}  remains valid under the hypotheses of  Theorems  \ref{EEB-thm} or \ref{EEB-thm2}.

To prove  Theorem  \ref{EEB-thm},  we use the decomposition \eqref{EEBeq3-1} with  $k_n= \lfloor n^{\beta}\rfloor$ and $      \max\{\frac{3}{\lambda}, \frac{4}{\eta} \}        <\beta < \frac{1}{6}$, and
  we divide the proof of \eqref{EEB-eq2.5}   into three lemmas.
\begin{lem}\label{EEB-lem11}
Under the hypothesis of  Theorem \ref{EEB-thm}, with $k_n= \lfloor n^{\beta}\rfloor$ and $      \max\{\frac{3}{\lambda}, \frac{4}{\eta} \}        <\beta < \frac{1}{6}$, we have
\begin{equation}\label{EEBeq3-3a}
{n}\mathbb{A}_n \xrightarrow{n \rightarrow \infty } 0 \quad  \mbox{ a.s.}
\end{equation}
\end{lem}
\begin{lem}\label{EEB-lem12} Under the hypothesis of  Theorem \ref{EEB-thm}, with $k_n= \lfloor n^{\beta}\rfloor$ and $      \max\{\frac{3}{\lambda}, \frac{4}{\eta} \}        <\beta < \frac{1}{6}$, we have, as $ n \rightarrow \infty$,
\begin{multline}\label{EEBeq3-4a}
 \mathbb{B}_n =\Bigg (\Phi(t)+\sum_{\nu=1}^{2}\frac{  Q_{\nu,n}(t)}{n^{\nu/2}} \Bigg) \overline{W}_{k_n}+\Big(-\frac{1}{s_n}\Big)\Bigg( \phi(t)+  \frac{  Q'_{1,n}(t)}{n^{1/2}}\Bigg)\overline{N}_{1,k_n}\\ +\frac{1}{2!}\frac{1}{s_n^2}  \phi'(t)  \overline{N}_{2,k_n}     +o\big(\frac{1}{n }\big)  \quad \mbox {a.s.},
\end{multline}
\end{lem}
\begin{lem}\label{EEB-lem13} Under the hypothesis of  Theorem \ref{EEB-thm},  with $k_n= \lfloor n^{\beta}\rfloor$ and $      \max\{\frac{3}{\lambda}, \frac{4}{\eta} \}        <\beta < \frac{1}{6}$, the following assertions hold a.s. as $ n \rightarrow \infty$:
\begin{eqnarray}
   \label{EEBeq3-8a} & & \overline{W}_{k_n}-W=o(\frac{1}{n}),  \\
    \label{EEBeq3-9a}& &  \overline{N}_{1,k_n}-V_1=o(\frac{1}{\sqrt{n}} ), \\
   \label{EEBeq3-10a}  &&    \overline{N}_{2,k_n}-V_2=o(1),
\end{eqnarray}
where $ \overline{W}_{k_n},\overline{N}_{1,k_n},\overline{N}_{2,k_n}$
 are defined  in  \eqref{EEBeq3-5},  \eqref{EEBeq3-6}.
\end{lem}

While in the proof of Theorem \ref{EEB-thm2}, we shall take $k_n= \lfloor n^{\beta}\rfloor$ with $\max\{\frac{4}{\lambda}, \frac{5}{\eta} \}<\beta < \frac{1}{8}$.
We still use  the decomposition \eqref{EEBeq3-1},   and  divide the proof of \eqref{EEB-eq2.8}   into three lemmas.
\begin{lem}\label{EEB-lem1}
Under the hypothesis of  Theorem \ref{EEB-thm2}, with $k_n= \lfloor n^{\beta}\rfloor$ and  $\max\{\frac{4}{\lambda}, \frac{5}{\eta} \}<\beta < \frac{1}{8}$, we have
\begin{equation}\label{EEBeq3-3}
{n^{3/2}}\mathbb{A}_n \xrightarrow{n \rightarrow \infty } 0 \quad  \mbox{ a.s.}
\end{equation}
\end{lem}
\begin{lem}\label{EEB-lem2} Under the hypothesis of  Theorem \ref{EEB-thm2}, with $k_n= \lfloor n^{\beta}\rfloor$ and  $\max\{\frac{4}{\lambda}, \frac{5}{\eta} \}<\beta < \frac{1}{8}$, the following holds a.s. as $ n \rightarrow \infty$:
\begin{multline}\label{EEBeq3-4}
 \mathbb{B}_n =\Bigg (\Phi(t)+\sum_{\nu=1}^{3}\frac{  Q_{\nu,n}(t)}{n^{\nu/2}} \Bigg) \overline{W}_{k_n}+\Big(-\frac{1}{s_n}\Big)\Bigg( \phi(t)+ \sum_{\nu=1}^{2}\frac{  Q'_{\nu,n}(t)}{n^{\nu/2}}\Bigg)\overline{N}_{1,k_n}\\ +\frac{1}{2!}\frac{1}{s_n^2} \Bigg( \phi'(t) + \frac{Q''_{1,n}(t)}{n^{1/2}}\Bigg)\overline{N}_{2,k_n} + \frac{1}{3!}(-\frac{1}{s_n^3}) \phi''(t)\overline{N}_{3,k_n}   +o\big(\frac{1}{n^{3/2}}\big),
\end{multline}
\end{lem}
\begin{lem}\label{EEB-lem3} Under the hypothesis of  Theorem \ref{EEB-thm2}, with $k_n= \lfloor n^{\beta}\rfloor$ and  $\max\{\frac{4}{\lambda}, \frac{5}{\eta} \}<\beta < \frac{1}{8}$,  the following assertions hold a.s. as $ n \rightarrow \infty$:
\begin{eqnarray}
   \label{EEBeq3-8} & & \overline{W}_{k_n}-W=o(\frac{1}{n^{3/2}}),  \\
    \label{EEBeq3-9}& &  \overline{N}_{1,k_n}-V_1=o( \frac{1}{n}),  \\
   \label{EEBeq3-10}  &&    \overline{N}_{2,k_n}-V_2=o(\frac{1}{\sqrt{n}}), \\
   \label{EEBeq3-11}  &&    \overline{N}_{3,k_n}-V_3=o(1),
\end{eqnarray}
where $ \overline{W}_{k_n},\overline{N}_{1,k_n},\overline{N}_{2,k_n},\overline{N}_{3,k_n}$
 are defined  in  \eqref{EEBeq3-5},  \eqref{EEBeq3-6} \eqref{EEBeq3-7}.\end{lem}

To avoid  repetition, here we shall only present  the proofs  of Lemmas \ref{EEB-lem1}, \ref{EEB-lem2} and \ref{EEB-lem3};  similar arguments apply  to Lemmas \ref{EEB-lem11}, \ref{EEB-lem12} and \ref{EEB-lem13}.

\subsection{Proofs of Lemmas  \ref{EEB-lem1}, \ref{EEB-lem2} and \ref{EEB-lem3}  }

\begin{proof}[Proof of Lemma \ref{EEB-lem1}]
The proof  is similar to that of Lemma 5.1 in \cite{GL14}.
For ease of  notation, we will denote by  $ [f(x)]_{x=a}$   the value of a function $f(x)$ at  the point $a$,    and  define for $|u|=k_n$,
\begin{align*}
    &  X_{n,u}= W_{n-k_n}(u,s_n t-S_u) - \E_{\xi, k_n}W_{n-k_n}(u,s_n t-S_u), ~~\bar{X}_{n,u} =  X_{n,u} \mathbf{1}_{\{|X_{n,u}| <\Pi_{k_n}\}}, \\
     & \bar{\mathbb{A}}_{n} =   \frac{1}{\Pi_{k_n}} \sum_{u\in \T_{k_n} } \bar{X}_{n,u}.
\end{align*}
Then we see that $ |X_{n,u}|\leq W_{n-k_n}(u)+1$.

To prove Lemma \ref{EEB-lem1}, we will use the extended Borel-Cantelli Lemma. We can obtain the required result once we prove that $\forall \varepsilon>0$,
\begin{equation}\label{eq10}
  \sum_{n=1}^{\infty} \P_{k_n} (|{n^{3/2}} \mathbb{A}_n| >2 \varepsilon) <\infty.
\end{equation}
Notice that \begin{eqnarray*}
              && \P_{k_n}(|\mathbb{A}_n| >\frac{2\varepsilon}{{n}^{3/2}})    \\
              &\leq&  \P_{k_n} (\mathbb{A}_n\neq \bar{\mathbb{A}}_n ) +\P_{k_n} (|\bar{\mathbb{A}}_n - \E_{\xi, k_n} \bar{\mathbb{A}}_n| > \frac{\varepsilon}{ {n}^{3/2}} )  +\P_{k_n} ( |\E_{\xi, k_n} \bar{\mathbb{A}}_n| >\frac{\varepsilon}{ {n}^{3/2}}).
            \end{eqnarray*}
Then we can  proceed  the proof  in  3 steps.\\
{\bf Step 1 } We first prove that
\begin{equation}\label{cbrweq3-7}
 \sum_{n=1}^{\infty}{\P}_{k_n} (\mathbb{A}_n\neq \overline{\mathbb{A}}_n) <\infty.
\end{equation}
To this end,  define  $$W^*=\sup_n W_n.$$
We need the following result  on $W^*$.
\begin{lem}[\cite{LiangLiu10}, Theorem 1.2] \label{lem5}
  Assume    \eqref{EEB-eq2.1}  for some $\lambda>0$  and $\E m_0^{-\delta}<\infty$ for some $\delta>0$.  Then
  \begin{equation}\label{cbrweq17a}
    {\E}(W^*+1)(\ln (W^*+1))^{\lambda} <\infty.
\end{equation}\end{lem} Observe that
\begin{align*}
      \P_{k_n}(\mathbb{A}_n\neq \overline{\mathbb{A}}_n )&\leq \sum_{u\in \T_{k_n} } \P_{k_n}(X_{n,u}\neq \overline{X}_{n,u} )  =  \sum_{u\in \T_{k_n} } \P_{k_n}(|X_{n,u}|\geq \Pi_{k_n}) \\&\leq \sum_{u\in \T_{k_n} }  \P_{k_n} ( W_{n-k_n} (u)+1 \geq \Pi_{k_n})\\ &=  W_{k_n} \Big[r_n \P( W_{n-k_n}+1 \geq r_n )\Big]_{r_n=\Pi_{k_n}}
   \\&\leq W_{k_n} \Big[\E\big((W_{n-k_n}+1 ) \mathbf{1}_{ \{W_{n-k_n}+1 \geq r_n\}} \big)\Big]_{r_n=\Pi_{k_n}}
  \\ &\leq W_{k_n} \Big[\E\big((W^*+1 ) \mathbf{1}_{ \{W^*+1 \geq r_n\}} \big)\Big]_{r_n=\Pi_{k_n}}
  \\   & \leq  W^*(\ln \Pi_{k_n})^{-\lambda}\E (W^*+1)(\ln (W^*+1))^{\lambda}
  \\ &\leq K_\xi W^* n^{-\lambda \beta} \E (W^*+1)(\ln (W^*+1))^{\lambda},
\end{align*}
where the last inequality holds since
\begin{equation}\label{cbrweq4.9}
\frac{1}{n}  \ln \Pi_{n} \rightarrow \E\ln m_0>0   \mbox{ a.s. },
\end{equation}
 and  $k_n\sim n^{\beta}$.
By the choice of $\beta$ and Lemma \ref{lem5}, we obtain  \eqref{cbrweq3-7}.

\medskip
\noindent {\bf Step 2}.  We next prove that $\forall \varepsilon>0$,
\begin{equation}\label{cbrweq3-8}
 \sum_{n=1}^{\infty}{\P}_{k_n} ( |\overline{\mathbb{A}}_n -{\E}_{\xi,k_n} \overline{\mathbb{A}}_n|>\frac{ \varepsilon}{{n}^{3/2}}) <\infty.
\end{equation}

Take  a constant $b \in  (1, e^{\E\ln m_0})$. Observe  that $\forall u\in \T_{k_n}, n\geq 1$,
\begin{eqnarray*}
   \E_{k_n} \bar{X}^2_{n,u} &=&  \int_{0}^\infty 2x\P_{k_n} (|\bar{X}_{n,u}|>x  )  dx
    =  2\int_0^\infty x\P_{k_n}  ( |{X}_{n,u} | \mathbf{1 }_{ \{|{X}_{n,u} |<\Pi_{k_n}  \}} >x ) dx\\
   &\leq & 2\int_0^{\Pi_{k_n}} x\P_{k_n} ( | W_{n-k_n}(u)+1 | >x) dx= 2\int_0^{\Pi_{k_n}} x\P ( | W_{n-k_n}+1 | >x) dx \\
    &\leq&  2 \int_0^{\Pi_{k_n} }  x\P( W^*+1 >x) dx \\
    & \leq& 2   \int_e^{\Pi_{k_n} } (\ln x)^{-\lambda} \E (W^*+1)(\ln(W^*+1) )^{\lambda}  dx + 2  \int_0^e xdx\\
    &\leq & 2 \E (W^*+1)(\ln(W^*+1) )^{\lambda} \left(\int_e^{b^{k_n}} (\ln x)^{-\lambda} dx + \int_{b^{k_n}}^{\Pi_{k_n}} (\ln x)^{-\lambda} dx \right)+e^2\\
    & \leq & 2 \E (W^*+1)(\ln(W^*+1) )^{\lambda}(b^{k_n}  + (\Pi_{k_n}-b^{k_n}) (k_n\ln b  )^{-\lambda} )+e^2.
\end{eqnarray*}
Then  we have that
\begin{align*}
&~ ~~\sum_{n=1}^{\infty} {\P}_{k_n}(|\overline{\mathbb{A}}_n-{\E}_{\xi,k_n} \overline{\mathbb{A}}_n|>\frac{\varepsilon}{{n}^{3/2}})\\  &=
 \sum_{n=1}^{\infty} {\E}_{k_n}{\P}_{\xi,k_n} (|\overline{\mathbb{A}}_n-{\E}_{\xi,k_n} \overline{\mathbb{A}}_n|>\frac{\varepsilon}{{n}^{3/2}}) \\
    &\leq \varepsilon^{-2}\sum_{n=1}^{\infty} n^3{\E}_{k_n}\left( {\Pi_{k_n}^{-2}}    \sum_{u\in \T_{k_n}}{\E}_{\xi,k_n} \overline{X}_{n,u}^2\right)= \varepsilon^{-2} \sum_{n=1}^{\infty}n^3 \left( {\Pi_{k_n}^{-2} }    \sum_{u\in \T_{k_n}}{\E}_{k_n} \overline{X}_{n,u}^2\right)\\
   & \leq \varepsilon^{-2} \sum_{n=1}^{\infty} \frac{n^3W_{k_n}}{\Pi_{k_n}}  \big [2 \E (W^*+1)(\ln(W^*+1)^{\lambda} )(b^{k_n}  + (\Pi_{k_n}-b^{k_n}) (k_n\ln b  )^{-\lambda} )+e^2\big ]\\
    & \leq  2\varepsilon^{-2}W^* \E (W^*+1)(\ln(W^*+1)^{\lambda} ) \bigg(  \sum_{n=1}^{\infty} \frac{n^3}{\Pi_{k_n}}b^{k_n} + \sum_{n=1}^{\infty} n^3 (k_n\ln b  )^{-\lambda}   \bigg) +e^2\varepsilon^{-2}W^*  \sum_{n=1}^{\infty} \frac{n^3}{\Pi_{k_n}}.
\end{align*}
By \eqref{cbrweq4.9} and $\lambda \beta >4$, the three series in the last expression above converge under our hypothesis and hence \eqref{cbrweq3-8} is proved.

\medskip
\noindent {\bf Step 3.} Observe
\begin{eqnarray*}
 & & \P_{k_n} \Bigg(| \E_{\xi,k_n} \bar{\mathbb{A}}_n | > \frac{\varepsilon}{{n}^{3/2}} \Bigg )  \\
   & \leq &  \frac{{n}^{3/2}}{\varepsilon}  \E_{k_n}  | \E_{\xi,k_n} \bar{\mathbb{A}}_n |
   = \frac{{n}^{3/2}}{\varepsilon} \E_{k_n} \Big|   \frac{1}{\Pi_{k_n} }   \sum_{u\in \T_{k_n}}  \E_{\xi,k_n} \bar{X}_{n,u} \Big|\\ & = &\frac{{n}^{3/2}}{\varepsilon} \E_{k_n} \Big|   \frac{1}{\Pi_{k_n} }   \sum_{u\in \T_{k_n}}  (- \E_{\xi,k_n } X_{n,u} \mathbf{1}_{ \{ |X_{n,u}| \geq \Pi_{k_n}\}} ) \Big|
   \\ & \leq &  \frac{{n}^{3/2}}{\varepsilon} \frac{1}{\Pi_{k_n} }   \sum_{u\in \T_{k_n}}  \E_{k_n } ( W_{n-k_n} (u) +1) \mathbf{1}_{ \{ W_{n-k_n}(u)+1 \geq \Pi_{k_n}\}} \\ & = &    \frac{{n}^{3/2}W_{k_n}}{\varepsilon} \Big[ \E ( W_{n-k_n} +1) \mathbf{1}_{ \{ W_{n-k_n}+1 \geq r_n\}} \Big]_{r_n =\Pi_{k_n}}\\ &\leq& \frac{W^*}{\varepsilon}{n}^{3/2} \Big[ \E ( W^* +1) \mathbf{1}_{ \{ W^*+1 \geq r_n\}} \Big]_{r_n =\Pi_{k_n}}\\
   &\leq &\frac{W^*}{\varepsilon}  \frac{{n}^{3/2}}{(\ln\Pi_{k_n})^{\lambda} } \E    (W^*+1)  \ln ^{\lambda } ( W^*+1 )
   \\ & \leq & \frac{W^*}{\varepsilon}  K_\xi n^{3/2 -\lambda\beta}  \E    (W^*+1)  \ln ^{\lambda } ( W^*+1 ).
 \end{eqnarray*}Then  by \eqref{cbrweq4.9} and  $ \lambda \beta >4$,   it follows that \[ \sum_{n=1}^\infty \P_{k_n} \Bigg(| \E_{\xi,k_n} \bar{\mathbb{A}}_n | > \frac{\varepsilon}{{n}^{3/2}} \Bigg )<\infty. \]

Combining Steps 1-3, we obtain \eqref{eq10}. Hence  the lemma is proved.

\end{proof}

\begin{proof}[Proof of Lemma \ref{EEB-lem2}]
For ease of reference, we introduce some notation:
\begin{align*}
  &\kappa_{1, n}   ={\frac{1}{6} (s_n^2-s_{k_n}^2 )^{-3/2} }{(s_n^{(3)} - s_{k_n}^{(3)}) }, ~\qquad D_1(x)=-H_2(x)\phi(x) , \\ & \kappa_{2, n}  =   \frac{1}{72} (s_n^2-s_{k_n}^2 )^{-3} (s_n^{(3)} - s_{k_n}^{(3)})^2, ~ \qquad D_2(x)= -H_5(x)\phi(x), \\
 & \kappa_{3, n}  =\frac{ 1}{24} (s_n^2-s_{k_n}^2 )^{-2} \sum_{j=k_n}^{n-1} (\sigma_j^{(4)} -3\big(\sigma_j^{(2)} \big)^2 ), \quad
  D_3(x)= -H_3(x) \phi(x), \\
&  R_n(x) =-\frac{\Big(s_n^{(3)}-s_{k_n}^{(3)} \Big)^3}{1296 (s_n^2-s_{k_n}^2)^{9/2}} H_8(x) \phi(x) -\frac{\sum_{j=k_n}^{n-1} \Big(\sigma_j^{(5)} -10\sigma_j^{(3)}\sigma_j^{(2)} \Big)}{120 (s_n^2-s_{k_n}^2)^{5/2}}H_4(x)  \phi(x)\\ &\qquad \qquad-\frac{ \Big(s_n^{(3)}-s_{k_n}^{(3)}\Big ) \sum_{j=k_n}^{n-1} \Big(\sigma_j^{(4)} -3\big(\sigma_j^{(2)}\big)^2  \Big)} { 144  (s_n^2-s_{k_n}^2)^{7/2}}H_6(x)  \phi(x),
\end{align*}
Observe that
\begin{equation}\label{EEBeq4-11}
   \mathbb{B}_n= \mathbb{B}_{n1}+ \mathbb{B}_{n2}+ \mathbb{B}_{n3},
\end{equation}
where \begin{align*}
 \mathbb{B}_{n1}    & =\frac{1}{\Pi_{k_n }}   \sum_{u\in \mathbb{T}_{k_n}}\ind{|S_u|> k_n} \Bigg [{\E}_{\xi,k_n} W_{n-k_n}(u,s_n t-S_u)
  \Bigg ],\\
      \mathbb{B}_{n2}  & = \frac{1}{\Pi_{k_n}} \sum_{u\in \T_{k_n} }\ind{|S_u|\leq k_n}   \Bigg[  \E_{\xi,k_n} W_{n-k_n} (u,s_n t-
S_u)- \Phi\bigg( \frac{s_n t-S_u}{( s_n^2-s_{k_n}^2)^{1/2}}\bigg)\\
& \qquad \qquad \qquad \quad\quad\quad-\sum_{\nu=1}^3\kappa_{\nu,n} D_{\nu}\bigg( \frac{s_n t-S_u}{( s_n^2-s_{k_n}^2)^{1/2}}\bigg)-R_n\bigg( \frac{s_n t-S_u}{( s_n^2-s_{k_n}^2)^{1/2}}\bigg)\Bigg],  \\
 \mathbb{B}_{n3}  & =  \frac{1}{\Pi_{k_n}} \sum_{u\in \T_{k_n} } \ind{|S_u|\leq k_n} \Bigg[   \Phi\bigg( \frac{s_n t-S_u}{( s_n^2-s_{k_n}^2)^{1/2}}\bigg)+\sum_{\nu=1}^3\kappa_{\nu,n} D_{\nu}\bigg( \frac{s_n t-S_u}{( s_n^2-s_{k_n}^2)^{1/2}}\bigg)\\
 & \qquad \qquad \qquad  \qquad \qquad \qquad + R_n\bigg( \frac{s_n t-S_u}{( s_n^2-s_{k_n}^2)^{1/2}}\bigg) \Bigg].
\end{align*}
The lemma will be proved once we show that  a.s.
\begin{align}\label{EEBeq3-12}
  & n^{3/2}\mathbb{B}_{n1} \xrightarrow{n\rightarrow\infty} 0,  \\
   \label{EEBeq3-13}
    &    n^{3/2}\mathbb{B}_{n2}   \xrightarrow{n\rightarrow\infty} 0,\\
 \nonumber   &  \mathbb{B}_{n3}=\Bigg (\Phi(t)+\sum_{\nu=1}^{3}\frac{  Q_{\nu,n}(t)}{n^{\nu/2}} \Bigg) \overline{W}_{k_n}+\Big(-\frac{1}{s_n}\Big)\Bigg( \phi(t)+ \sum_{\nu=1}^{2}\frac{  Q'_{\nu,n}(t)}{n^{\nu/2}}\Bigg)\overline{N}_{1,k_n}\\  \label{EEBeq3-14} & \qquad \qquad +\frac{1}{2!}\frac{1}{s_n^2} \Bigg( \phi'(t) + \frac{Q''_{1,n}(t)}{n^{1/2}}\Bigg)\overline{N}_{2,k_n} + \frac{1}{3!}(-\frac{1}{s_n^3}) \phi''(t)\overline{N}_{3,k_n}   +o\big(\frac{1}{n^{3/2}}\big),
\end{align}
where $ \overline{W}_{k_n},\overline{N}_{1,k_n},\overline{N}_{2,k_n}, \overline{N}_{3,k_n}$  are defined in \eqref{EEBeq3-5}--\eqref{EEBeq3-7}.
We will prove these results subsequently.

 First we prove \eqref{EEBeq3-12}.
 Since          \begin{equation*}
           |\mathbb{B}_{n1}| \leq   \frac{1}{\Pi_{k_n}} \sum_{u\in \T_{k_n} } \ind{|S_u|> k_n} ,
          \end{equation*}
 it will follow from the following fact:
 \begin{equation}\label{EEB-eq3.13}
   {n}^{3/2}  \frac{1}{\Pi_{k_n}} \sum_{u\in \T_{k_n}} \mathbf{1}_{\{S_u|>k_n\}} \xrightarrow{n\rightarrow \infty} 0 ~ \mbox{a.s.}
 \end{equation}
 In order to prove \eqref{EEB-eq3.13}, we first observe that
\begin{align*}
&
   \E \left( \sum_{n=1}^\infty {n}^{3/2} \frac{1}{\Pi_{k_n}} \sum_{u\in \T_{k_n}} \ind{|S_u|>k_n} \right)
    =\sum_{n=1}^\infty  {n}^{3/2} \E  \ind{|\widehat{S}_{k_n}|>k_n }  \leq\ \sum_{n=1}^\infty  {n} ^{3/2} k_n^{-\eta} \E|\widehat{S}_{k_n}|^{\eta}
  \\ & \leq\sum_{n=1}^\infty  {n}^{3/2} k_n^{-\frac{\eta}{2}-1}  \sum_{j=0}^{k_n-1} \E |\widehat{L}_j|^{\eta}
           =\sum_{n=1}^\infty n^{3/2} k_n^{-\frac{\eta}{2}}\E |\widehat{L}_0|^{\eta},
\end{align*}
where $\widehat{S}_{k_n}= \sum_{j=0}^{k_n-1} \widehat{L}_j$.
By the choice of $\beta$ and  $k_n$,  $3/2-\beta \eta/2  <-1$ and
the series in the right hand side of the  above expression converges.
So $$
\sum_{n=1}^\infty n^{3/2} \frac{1}{\Pi_{k_n}} \sum_{u\in \T_{k_n}} \ind{|S_u|>k_n} <\infty\mbox{ ~  a.s.,} $$
 which implies \eqref{EEB-eq3.13}, and  consequently  \eqref{EEBeq3-12} follows.

The proof of \eqref{EEBeq3-13} will mainly be based on  the following result about the asymptotic expansion of the distribution of the sum of random variables.
\begin{prop}\label{EEBpro2} Under the hypothesis of Theorem \ref{EEB-thm2}, for a.e. $\xi$,
\begin{equation*}
\varepsilon_n := n^{\frac{3}{2}} \sup_{x\in \R} \Bigg|  \P_{\xi } \bigg(\frac{\sum_{k=k_n}^{n-1}\widehat{L}_{k}}{( s_n^2-s_{k_n}^2)^{1/2}} \leq x \bigg)-  \Phi (x) -\sum_{\nu=1}^3\kappa_{\nu,n}D_\nu(x)   -R_n(x)\Bigg|  \xrightarrow{n\rightarrow\infty}  0.
\end{equation*}
\end{prop}
\begin{proof}
Denote by $v_k = v(\xi_k) $ the characteristic function of the random distribution $G(\xi_k)$, which is also the characteristic function of $\widehat{L}_{k}$ under $\P_\xi$: for all real $t$,
 $ v_k(t)  = \int e^{itx} G(\xi_k) (dx) = \E_\xi e^{it \widehat{L}_{k}} $.  Combining the Markov inequality with Lemma \ref{lem-Edge-exp}, we obtain the following result:
\begin{align*}
    &\sup_{x\in \R} \Bigg|\P_{\xi } \bigg(\frac{\sum_{k=k_n}^{n-1}\widehat{L}_{k}}{( s_n^2-s_{k_n}^2)^{1/2}} \leq x \bigg)-  \Phi (x) -\sum_{\nu=1}^3\kappa_{\nu,n}D_\nu(x) -R_n(x)   \Bigg|    \\
\leq  & K_\xi  \left\{ (s_n^2-s_{k_n}^2 )^{-3} \sum_{j=k_n }^{n-1} \E_\xi | \widehat{L}_{j}|^6+  n^{15} \left(\sup_{|t| >T } \frac{1}{n} \bigg( k_n+ \sum_{j=k_n}^{n-1} |v_j(t)| \bigg)+ \frac{1}{2n}\right)^n\right\}.
\end{align*}
By our conditions on the environment, we know that
\begin{equation}\label{EEB-eq3.14}
\lim_{n\rightarrow \infty} n^{2}  {(s_n^2-s_{k_n}^2)^{- 3}} \sum_{j=k_n}^{n-1} \E_\xi |\widehat{L}_k|^6  = \E |\widehat{L}_0|^6/ (\E \sigma_0^{(2)})^{3}.
\end{equation}
By \eqref{EEB-eq2.3}, $v_0 $
satisfies
\begin{equation*}
 \P\Big(  \limsup_{|t|\rightarrow\infty}|v_0(t)|<1 \Big) >0.
\end{equation*}
So there exist   constants $T>0$ and  $0<c<1$ such that $ \P\Big(\sup_{|t|> T} |v_0(t)|<c \Big) >0$. Since $v_n$ has the same law as $v_0$, it follows that $ \P\Big(\sup_{|t|> T} |v_n(t)|<c \Big) >0$. Define $c(\xi_n) = c$ if the characteristic function $v_n = v(\xi_n)  $ of
$G(\xi_n)$ satisfies $  \sup_{|t|> T} |v_n(t)|<c $, and
$c(\xi_n) = 1$ otherwise. Then
$c_n := c(\xi_n) $ satisfies $0< c_n \leq 1$ (in fact $c_n = c$ or $1$),
\begin{equation*}
  \sup_{|t|> T} |v_n(t)| \leq c_n\quad   \mbox{ and }  \quad  \P(c_n <1) >0.
 \end{equation*}
Consequently, by the law of large numbers, we have  \begin{align*}
                   \sup_{|t|>T}   \Big(\frac{1}{n}\sum_{j=k_n}^{n-1} |v_j(t)| \Big)&\leq  \frac{1}{n}\sum_{j=1}^{n-1} c_j \rightarrow \E c_0<1.
                                 \end{align*}
Then for $n$ large enough,
\begin{equation}\label{EEB-eq3.15}
  \bigg(\sup_{|t| >T } \frac{1}{n} \Big( k_n+ \sum_{j=k_n}^{n-1}  |v_j(t)| \Big)+ \frac{1}{2n}\bigg)^n=o( n^{-m}), \quad  \forall  m >0.
\end{equation}
The proposition comes from  \eqref{EEB-eq3.14} and \eqref{EEB-eq3.15}.
\end{proof}
Observe that for $u\in \T_{k_n}$,
$$ \E_{\xi,k_n} W_{n-k_n}(u,  s_nt-S_u)   =  \P_{\xi } \bigg(\frac{\sum_{k=k_n}^{n-1}\widehat{L}_{k}}{( s_n^2-s_{k_n}^2)^{1/2}} \leq x \bigg)\bigg |_{x=s_nt-S_u}.$$
From Proposition \ref{EEBpro2}, it follows that
\begin{equation}\label{EEB-eq3.16}
   n^{3/2}|\mathbb{B}_{n2}|\leq W_{k_n} \varepsilon_n  \xrightarrow{n\rightarrow\infty} 0.
\end{equation}
Hence \eqref{EEBeq3-13} is proved.

It remains to prove \eqref{EEBeq3-14}.   Our arguments  will depend heavily on Taylor's  expansion with tedious calculus. In the following,
 we shall use the notation $\varepsilon_n^*$ to denote an infinitesimal (which may change from line to line) dominated by another one $a_n$   depending only on the environment $\xi$ and on the value of $t$: that is
\begin{equation}\label{EEBequ}
 |\varepsilon_n^*| \leq a_n = a_n(\xi, t) \longrightarrow  0  \quad \mbox{as}  \quad n\rightarrow \infty.
\end{equation}
Below we suppose always that  $  u\in \T_{k_n}$ and $|S_u| \leq k_n$. Then
\begin{eqnarray*}
    \frac{s_nt-S_u}{\sqrt{s_n^2-s_{k_n}^2}} -t &=&  \bigg[\Big(1-\frac{s^2_{k_n}}{s_n^2}\Big)^{-1/2} -1\bigg]t-
    \bigg(1-\frac{s^2_{k_n}}{s_n^2}\bigg)^{-1/2}\frac{S_u}{s_n}   \\
  &=&  \bigg[1+  \frac{s^2_{k_n}}{2s_n^2}+  \varepsilon_n^* n^{-3/2}  -1  \bigg]t-  \bigg[1+  \frac{s^2_{k_n}}{2s_n^2}+ \varepsilon_n^* n^{-3/2}  \bigg] \frac{S_u}{s_n}\\
   &=&   -\frac{S_u}{s_n}+ \frac{s^2_{k_n}t}{2s_n^2}- \frac{s^2_{k_n}S_u}{2s_n^3}+ \varepsilon_n^* n^{-3/2}
\end{eqnarray*}
Further,  it is easy to see that
\begin{eqnarray*}
    \Bigg(\frac{s_nt-S_u}{\sqrt{s_n^2-s_{k_n}^2}} -t\Bigg)^2& =&  \frac{S_u^2}{s_n^2}-\frac{
    {  s_{k_n}^2} S_u t}{s_n^3} + \varepsilon_n^* n^{-3/2} ; \\
   \Bigg (\frac{s_nt-S_u}{\sqrt{s_n^2-s_{k_n}^2}} -t\Bigg)^3&=&  - \frac{S_u^3}{s_n^3}+ \varepsilon_n^* n^{-3/2} .
\end{eqnarray*}
By  Taylor's expansion and the above estimates,
\begin{eqnarray}
 \nonumber && \Phi\bigg( \frac{s_n t-S_u}{( s_n^2-s_{k_n}^2)^{1/2}}\bigg)
 \\   \nonumber & =& \Phi(t)+ \sum_{j=1}^3\frac{1}{j!}\Phi^{(j)}(t) \bigg(\frac{s_nt-S_u}{\sqrt{s_n^2-s_{k_n}^2}} -t\bigg)^j +\varepsilon_n^* n^{-3/2}   \\
\nonumber &=&\Phi(t) -\frac{1}{s_n} \phi(t)S_u - \frac{1}{2s_n^2}t\phi(t)(S_u^2-s_{k_n}^2)-\frac{1}{6s_n^3}\phi(t)H_2(t)( S_u^3-3s_{k_n}^2S_u )\\\label{EEBeq3-15}
  &&+\varepsilon_n^* n^{-3/2}.
\end{eqnarray}

Since \begin{eqnarray*}
       \kappa_{1,n}& =& {\frac{1}{6s_n^3} \bigg(1-\frac{s_{k_n}^2}{s_n^2} \bigg)^{-3/2} }{(s_n^{(3)} - s_{k_n}^{(3)}) }=  \frac{1}{6 s_n^{3}} \bigg(1+\frac{3s_{k_n}^2}{2s_n^2} + \varepsilon_n^*  n^{-1}   \bigg)(s_n^{(3)} - s_{k_n}^{(3)})\\
         &=&\frac{1}{6 s_n^{ 3}}s_n^{(3)} - \frac{1}{6 s_n^{ 3}} s_{k_n}^{(3)}  +\frac{ s_n^{(3)} s_{k_n}^2}{4s_n^5}+ \varepsilon_n^* n^{-3/2}
          \end{eqnarray*}
         and
          \begin{eqnarray*}
       &&  D_1(\frac{s_nt-S_u}{{(s_n^2-s_{k_n}^2)^{1/2}}}) \\
       & =& D_1(t)+ D_1'(t) \bigg(\frac{s_nt-S_u}{\sqrt{s_n^2-s_{k_n}^2}} -t\bigg) + \frac{1}{2}D_1''(t) \bigg(\frac{s_nt-S_u}{\sqrt{s_n^2-s_{k_n}^2}} -t\bigg)^2+  \varepsilon_n^*    n^{-1}
    \\   &=& -H_2(t)\phi(t) +H_3(t)\phi(t) (  -\frac{S_u}{s_n}+ \frac{s^2_{k_n}t}{2s_n^2}) -\frac{1}{2 s_n^2}H_4(t)\phi(t)  {S_u^2}+ \varepsilon_n^*    n^{-1} ,
                \end{eqnarray*}
 we obtain
 \begin{eqnarray}
   \nonumber &&   \kappa_{1,n} D_1(\frac{s_nt-S_u}{\sqrt{s_n^2-s_{k_n}^2}})\\ \nonumber&=&  -\frac{1}{6 s_n^{ 3}}s_n^{(3)}H_2(t)\phi(t) +\frac{1}{6 s_n^{ 3}} s_{k_n}^{(3)}H_2(t)\phi(t)- \frac{1}{6s_n^4} s_n^{ (3)}S_uH_3(t)\phi(t)\\
 \nonumber    & &-\frac{ s_n^{(3)} s_{k_n}^2}{4s_n^5}H_2(t)\phi(t)+\frac{1}{12s_n^5} s_n^{(3)} s_{k_n}^2tH_3(t)\phi(t)-\frac{1}{12s_n^5}s_n^{(3)}{S_u^2}H_4(t)\phi(t)+\varepsilon_n^* n^{-3/2} \\
   \nonumber &=&-\frac{1}{6 s_n^{ 3}}s_n^{(3)}H_2(t)\phi(t) +\frac{1}{6 s_n^{ 3}} s_{k_n}^{(3)}H_2(t)\phi(t)- \frac{1}{6s_n^4} s_n^{ (3)}S_uH_3(t)\phi(t) \\
    \label{EEBeq3-16}&& -\frac{1}{12s_n^5}s_n^{(3)}(S_u^2-s_{k_n}^2)H_4(t)\phi(t)+ \varepsilon_n^* n^{-3/2} ,
 \end{eqnarray}
where in the last step we   use  the recurrence relation of Hermite polynomials:
\begin{equation}\label{EEBHermite}
   H_{m+1}(t)=tH_m(t)-mH_{m-1}(t).
\end{equation}

Noticing  that
\begin{eqnarray*}
        \kappa_{2,n} &=&\frac{1}{72}\frac{(s_n^{(3)}-s_{k_n}^{(3)} )^2}{ (s_n^2-s_{k_n}^2)^3}=\frac{1}{72 s_n^6 } {(s_n^{(3)})^2} + \varepsilon_n^* n^{-3/2} , \\
         D_2(\frac{s_nt-S_u}{\sqrt{s_n^2-s_{k_n}^2}})&=&   D_2(t)+ D_2'(t) \bigg(\frac{s_nt-S_u}{\sqrt{s_n^2-s_{k_n}^2}} -t\bigg)+\varepsilon_n^*  n^{-\frac{1}{2}} \\
         &=& -H_5(t)\phi(t)+ H_6(t)\phi(t)(-\frac{1}{s_n}S_u)+\varepsilon_n^* n^{-1/2} ,
     \end{eqnarray*}
we have \begin{multline}\label{EEBeq3-17}
          \kappa_{2,n}D_2(\frac{s_nt-S_u}{\sqrt{s_n^2-s_{k_n}^2}})=
 \\ -\frac{1}{72 s_n^6 } {(s_n^{(3)})^2} H_5(t)\phi(t)-\frac{1}{72 s_n^7 } {(s_n^{(3)})^2} S_u H_6(t)\phi(t)+\varepsilon_n^* n^{-3/2} .
        \end{multline}
Observing that
\begin{eqnarray*}
  \kappa_{3,n}& =& \frac{ 1}{24} (s_n^2-s_{k_n}^2 )^{-2} \sum_{j=k_n}^{n-1} (\sigma_j^{(4)} -3\big(\sigma_j^{(2)} \big)^2 )\\& = & \frac{ 1}{24s_n^4} \sum_{j=0}^{n-1} (\sigma_j^{(4)} -3\big(\sigma_j^{(2)} \big)^2 )+  \varepsilon_n^* n^{-3/2} , \\
 D_3\bigg(\frac{s_nt-S_u}{\sqrt{s_n^2-s_{k_n}^2}}\bigg) &=&D_3(t) + D_3'(t)(-\frac{1}{s_n} S_u)+   \varepsilon_n^* n^{-1/2} \\
  & =&-H_3(t)\phi(t)-\frac{1}{s_n}H_4(t)\phi(t) S_u+  \varepsilon_n^* n^{-1/2} ,
\end{eqnarray*}
we get
   \begin{multline}\label{EEBeq3-18}
      \kappa_{3,n}D_3\bigg(\frac{s_nt-S_u}{\sqrt{s_n^2-s_{k_n}^2}}\bigg)
     = -\frac{ 1}{24s_n^4} \sum_{j=0}^{n-1} \Big(\sigma_j^{(4)} -3\big(\sigma_j^{(2)} \big)^2 \Big)H_3(t)\phi(t)\\
     -\frac{ 1}{24s_n^5} \sum_{j=0}^{n-1}  \Big(\sigma_j^{(4)} -3\big(\sigma_j^{(2)} \big)^2\Big ) H_4(t)\phi(t) S_u+  \varepsilon_n^* n^{-3/2} .
   \end{multline}
It is easy to check that
\begin{multline}\label{EEBeq3-19}
  R_n(\frac{s_nt-S_u}{\sqrt{s_n^2-s_{k_n}^2}})=-\frac{\Big(s_n^{(3)} \Big)^3}{1296 s_n^{9}} H_8(t) \phi(t) -\frac{1}{120 s_n^{5}}\sum_{j=0}^{n-1} \Big(\sigma_j^{(5)} -10\sigma_j^{(3)}\sigma_j^{(2)} \Big)H_4(t)  \phi(t)\\ -\frac{ s_n^{(3)} } { 144  s_n^{7}} \sum_{j=0}^{n-1} \Big(\sigma_j^{(4)} -3\big(\sigma_j^{(2)}\big)^2  \Big)H_6(t)  \phi(t) + \varepsilon_n^* n^{-3/2} =   \frac{Q_{3,n}(t)}{n^{\frac{3}{2}}} + \varepsilon_n^* n^{-3/2}.
  \end{multline}

Plugging  the expansions \eqref{EEBeq3-15}, \eqref{EEBeq3-16},  \eqref{EEBeq3-17}, \eqref{EEBeq3-18} and \eqref{EEBeq3-19} into  $\mathbb{B}_{n3}$ defined in \eqref{EEBeq4-11}, we deduce that
\begin{multline*}
   \Bigg|\mathbb{B}_{n3}-\Bigg (\Phi(t)+\sum_{\nu=1}^{3}\frac{  Q_{\nu,n}(t)}{n^{\nu/2}} \Bigg) \overline{W}_{k_n}+\Big(-\frac{1}{s_n}\Big)\Bigg( \phi(t)+ \sum_{\nu=1}^{2}\frac{  Q'_{\nu,n}(t)}{n^{\nu/2}}\Bigg)\overline{N}_{1,k_n}\\    +\frac{1}{2!}\frac{1}{s_n^2} \Bigg( \phi'(t) + \frac{Q''_{1,n}(t)}{n^{1/2}}\Bigg)\overline{N}_{2,k_n} + \frac{1}{3!}(-\frac{1}{s_n^3}) \phi''(t)\overline{N}_{3,k_n} \Bigg| \leq a_n \overline{W}_{k_n}  n^{-\frac{3}{2}}.
\end{multline*}
By
using \eqref{EEBequ} and the fact that
$$ 0\leq \overline{W}_{k_n} \leq W_{k_n}\xrightarrow[a.s.]{n \rightarrow \infty } W , $$
we obtain the desired \eqref{EEBeq3-14}.

The assertion \eqref{EEBeq3-4}  follows from \eqref{EEBeq3-12}, \eqref{EEBeq3-13}  and \eqref{EEBeq3-14},   hence  the lemma is proved.
\end{proof}

\begin{proof}[Proof of Lemma  \ref{EEB-lem3}]
Observe
\begin{equation*}
   \overline{W}_{k_n} -W = -\frac{1}{\Pi_{k_n}} \sum_{u\in \T_{k_n}} \mathbf{1}_{\{|S_u|>k_n\}} + (W_{k_n}-W).
\end{equation*}
Thus
\eqref{EEBeq3-8}    follows from \eqref{EEB-eq3.13}  and the following lemma.
\begin{lem}[\cite{HuangLiu}]\label{EEB-lem5}
Assume the condition  \eqref{EEB-eq2.1}.
 Then
\begin{equation*}
W-W_n=o(n^{-\lambda})\qquad a.s.
\end{equation*}
\end{lem}

Similarly,   \eqref{EEBeq3-9},\eqref{EEBeq3-10},   \eqref{EEBeq3-11}  follow from  Theorem \ref{EEB-prop}  and the following results:
 \begin{eqnarray}
           & &   {n}  \frac{1}{\Pi_{k_n}} \sum_{u\in \T_{k_n}} S_u \mathbf{1}_{\{|S_u|>k_n\}} \xrightarrow{n\rightarrow \infty} 0 ~ \mbox{a.s.};  \\
          & &  {n}^{1/2}  \frac{1}{\Pi_{k_n}} \sum_{u\in \T_{k_n}} (S_u^2-s_{k_n}^2) \mathbf{1}_{\{|S_u|>k_n\}} \xrightarrow{n\rightarrow \infty} 0 ~ \mbox{a.s.}, \\
          &&    \frac{1}{\Pi_{k_n}} \sum_{u\in \T_{k_n}} (S_u^3-3S_us_{k_n}^2- s_{k_n}^{(3)}) \mathbf{1}_{\{|S_u|>k_n\}} \xrightarrow{n\rightarrow \infty} 0 ~ \mbox{a.s.},
       \end{eqnarray}
which can be easily proved by following the lines of the proof of  \eqref{EEB-eq3.13}.
\end{proof}

\section{Convergence rates of the relevant martingales   }\label{EEB-sec5}
In this section, we shall prove Theorem \ref{EEB-prop}.
Recall that we assume throughout the article that $l_n=0$.  Then the martingales  reduce to  the following simplified versions:
\begin{eqnarray*}
    & &N_{1,n} = \frac{1}{\Pi_n} \sum_{u\in \T_n}  S_u ;   \\
    & & N_{2,n} = \frac{1}{\Pi_n} \sum_{u\in \T_n} \Big( S_u ^2 -s_n^2 \Big);\\
    && N_{3,n} = \frac{1}{\Pi_n} \sum_{u\in \T_n} \Big(S_u^3 -3S_u  s_n^2-  s_n^{(3)}   \Big).
\end{eqnarray*}
It is easy to verify that they are martingales with respect to the filtration $\D_n$, and we omit the details (see \cite{GL14}).

We shall only offer the detailed proof of part (3),  as parts (1) and (2)  will follow  by the same way with minor changes.

   The proof is adapted from Asmussen(1976, \cite{Asmussen1976AOP}).   The key idea is to find a proper truncation to show the convergence of the series $\sum_{n} a_n (N_{3,n+1} -N_{3,n})$ with  suitable $a_n$, which gives the information on the convergence rate of $\sum_{n=\kappa}^{\infty} N_{3,n} $.    The proof relies on the following lemma.
\begin{lem}[\cite{Asmussen1976AOP}, Lemma 2] \label{EEB-lem4}
   Let $ \{\alpha_n, \beta_n, n\geq 1  \} $ be sequences of real numbers.
  If $0<\alpha_n\nearrow \infty, $ and the series $\sum_{n=1}^\infty  \alpha_n \beta_n$ converges, then
  \begin{equation*}
     \sum_{n=\kappa}^\infty \beta_n= o\Big(\frac{1}{\alpha_\kappa}\Big).
  \end{equation*}
\end{lem}
\begin{proof}[Proof of Part (3) in Theorem \ref{EEB-prop}]
We  begin by introducing some notation:
\begin{eqnarray*}
    & & \lambda_\delta= \lambda-3-\delta, \quad  X_u=S_u^3- 3S_us_n^2-s_{n}^{(3)}\qquad  \mbox{  for   }  u\in \T_n,  \\
    & &    \textsc{I}_n= N_{3,n+1}-N_{3,n} = \frac{1}{\Pi_n} \sum_{u\in \T_n}\bigg ( \frac{1}{m_n} \sum_{i=1}^{N_u} X_{ui}   -X_u\bigg) ,
      \\  &&   \textsc{I}_n' =  \frac{1}{\Pi_n} \sum_{u\in \T_n}\bigg ( \frac{1}{m_n} \sum_{i=1}^{N_u} X_{ui}   -X_u\bigg) \mathbf{1}_{\{N_u\leq  \Pi_n/n^{\lambda_\delta} \}} .
\end{eqnarray*}

If we can prove that the series
 \begin{equation}\label{EEBeq4.1}
    \sum_{n=1}^\infty n^{\lambda_\delta} \textsc{I}_n     \quad \mbox{ converges a.s. }
 \end{equation}
then  by setting  $V_3=  \sum_{n=1}^\infty \textsc{I}_n + \textsc{I}_1$    and using Lemma \ref{EEB-lem4}, we obtain  the desired conclusion.

We  prove  \eqref{EEBeq4.1}  by showing the following three series converge:
\begin{equation}\label{EEBeq4.2}
 \sum_{n=1}^\infty n^{\lambda_\delta} (\textsc{I}_n-\textsc{I}_n'
 ), \quad   \sum_{n=1}^\infty n^{\lambda_\delta} (\textsc{I}_n'- \E_{\xi,n}\textsc{I}_n'
 ), \quad \sum_{n=1}^\infty n^{\lambda_\delta}  \E_{\xi,n}\textsc{I}_n'.
\end{equation}
By  using an inequality  for moment of sums of independent random variables with mean zero,  it is easy to see that  for $  \widehat{S}_n=\sum_{j=0}^{n-1}\widehat{L}_j$,
\begin{equation}\label{EEBeq4.3}
    \E_\xi |\widehat{S}_n|^r \leq n^{\frac{r}{2}-1}   \sum_{j=0}^{n-1} \E_\xi |\widehat{L}_j|^r \leq K_\xi n^{\frac{r}{2}};
\end{equation} whence for $|u|=n$,
\begin{equation}\label{EEBeq4.4}
 \E_\xi | X_u| \leq  K_\xi n^{3/2};   \quad \E_\xi | X_u|^2 \leq  K_\xi n^3.
\end{equation}

For the first series in  \eqref{EEBeq4.2},   we observe that
\begin{eqnarray*}
  &\E_\xi|I_n-I_n'|&\leq \frac{1}{\Pi_n} \E_\xi \sum_{u\in\T_n}  \Bigg| \frac{1}{m_n} \sum_{i=1}^{N_u} X_{ui}   -X_u\Bigg| \mathbf{1}_{\{ {N_u}/{m_n}>  n^{-\lambda_\delta}\Pi_n \}}   \\
  &&\leq \frac{K_\xi n^3}{\Pi_n} \E_\xi \sum_{u\in\T_n}   (N_u/m_n +1)  \mathbf{1}_{\{{N_u}/{m_n}>  n^{-\lambda_\delta}\Pi_n \}}  \\
    &&=  K_\xi n^3 \E_\xi (\widehat{N}_n /m_n+1)  \mathbf{1}_{\{{\widehat{N}_n}/{m_n} >n^{-\lambda_\delta}\Pi_n\}}
    \\ &&\leq K_\xi n^3  \frac{1}{ \ln^{1+\lambda}
    ( \Pi_n/n^{\lambda_\delta}) }\E_\xi (\widehat{N}_n /m_n+1)  \ln^{1+\lambda} \big (\widehat{N}_n/m_n\big)
    \\ & &\leq_\eqref{cbrweq4.9}K_\xi n^{2-\lambda} \E_\xi (\widehat{N}_n /m_n+1)  \Big(\ln^+\widehat{N}_n\Big)^{1+\lambda} +    K_\xi n^{2-\lambda}    (\ln^- m_n)^{1+\lambda}
    \mbox{ a.s. }
\end{eqnarray*}
We see that
\begin{align*}
                                      & \E \sum_{n=1}^\infty {n^{\lambda_\delta+2-\lambda} } \bigg[ \E_\xi (\widehat{N}_n /m_n+1)  \Big(\ln^+\widehat{N}_n\Big)^{1+\lambda} + (\ln^- m_n)^{1+\lambda}\bigg ]
                                   \\ = &   \sum_{n=1}^\infty {n^{-1-\delta} } \bigg[ \E(\widehat{N}_0 /m_0+1)  \Big(\ln^+\widehat{N}_0\Big)^{1+\lambda} + \E(\ln^- m_0)^{1+\lambda}\bigg ]   <\infty,
                                  \end{align*}
which implies that
\begin{equation}\label{EEBeq4.5}
  \sum_{n=1}^\infty {n^{\lambda_\delta+2-\lambda} }\bigg[\E_\xi (\widehat{N}_n /m_n+1)  \Big(\ln^+\widehat{N}_n\Big)^{1+\lambda}+ (\ln^- m_n)^{1+\lambda}\bigg ] <\infty \quad   a.s.
\end{equation}
Thus
\begin{align*}
    &  \E_{\xi} \bigg|\sum_{n=1}^\infty n^{\lambda_\delta} (\textsc{I}_{ n}-\textsc{I}_{n}') \bigg | \leq \sum_{n=1}^\infty n^{\lambda_\delta}\E_{\xi}|\textsc{I}_{n}-\textsc{I}_{n}'| <\infty,  \\
   &      \E_{\xi}   \bigg| \sum_{n=1}^\infty   n^{\lambda_\delta}\E_{\xi,n} \textsc{I}_{n}'  \bigg|=\E_{\xi}  \bigg | \sum_{n=1}^\infty  n^{\lambda_\delta} \E_{\xi,n}  (\textsc{I}_{n}- \textsc{I}_{n}')  \bigg| \leq  \sum_{n=1}^\infty n^{\lambda_\delta} \E_{\xi}|\textsc{I}_{n}-\textsc{I}_{n}'| <\infty.
\end{align*}
It follows that  the series $\sum_{n=1}^\infty n^{\lambda_\delta} (\textsc{I}_{ n}-\textsc{I}_{ n}') $  and $\sum_{n=1}^\infty   n^{\lambda_\delta}\E_{\xi,n} \textsc{I}_{ n}' $ converge a.s.

It remains to prove the a.s. convergence of $\sum_{n=1}^\infty n^{\lambda_\delta} (\textsc{I}_n'- \E_{\xi,n}\textsc{I}_n'
 ) $.
By using the fact  that $\sum_{k=1}^{n} k^{\lambda_\delta} (\textsc{I}_k'- \E_{\xi,k}\textsc{I}_k'
 )    $  is a martingale with respect to $\{\D_{n+1}\}$ and by  the a.s. convergence of an  $L^2$ bounded martingale (see e.g. \cite[P. 251, Ex. 4.9]{Durrett96Proba}),  we  only need to show
that   the series
$$ \sum_{n=1}^\infty   n^{2\lambda_\delta} \E_\xi(\textsc{I}_n'- \E_{\xi,n}\textsc{I}_n'
 )^2  \quad   \mbox { converges a.s.}$$
To this end,  we first note that
\begin{eqnarray*}
    &&\qquad   \E_{\xi}\Bigg[\bigg ( \frac{1}{m_n} \sum_{i=1}^{N_u} X_{ui}   -X_u\bigg) ^2 \mathbf{1}_{\{ {N_u}/{m_n}\leq  n^{-\lambda_\delta}\Pi_n \}} \Bigg| \F_n\Bigg]
    \\&&=  \E_\xi \Bigg\{ \E_\xi \bigg[  \bigg ( \frac{1}{m_n} \sum_{i=1}^{N_u} X_{ui}   -X_u\bigg) ^2  \bigg|N_u\bigg]\mathbf{1}_{\{ {N_u}/{m_n}\leq  n^{-\lambda_\delta}\Pi_n \}} \Bigg\}
    \\ &&\leq\E_\xi \Bigg\{ 2  \bigg( N_u \sum_{i=1}^{N_u} \frac{\E_\xi X_{ui}^2}{{m_n^2}} + \E_\xi X_u^2\bigg)  \mathbf{1}_{\{ {N_u}/{m_n}\leq  n^{-\lambda_\delta}\Pi_n \}} \Bigg\}
    \\ &&\leq_{\eqref{EEBeq4.4}} K_\xi n^3\bigg(\E_\xi  \frac{N_u^2}{m_n^2}\mathbf{1}_{\{ {N_u}/{m_n}\leq  n^{-\lambda_\delta}\Pi_n \}} +1 \bigg )
     \\ &&= K_\xi n^3\bigg(\E_\xi  \frac{\widehat{N}_n^2}{m_n^2}\mathbf{1}_{\{ {\widehat{N}_n}/{m_n}\leq  n^{-\lambda_\delta}\Pi_n \}} +1 \bigg ).
\end{eqnarray*}
We next observe that
\begin{eqnarray*}
   && \quad n^{2\lambda_\delta}\E_\xi(\textsc{I}_n'- \E_{\xi,n}\textsc{I}_n'
 )^2 \\
 &&=  \frac{n^{2\lambda_\delta}}{\Pi_n^2} \E_\xi  \sum_{u\in\T_n}   \E_{\xi,n}\Bigg[\bigg ( \frac{1}{m_n} \sum_{i=1}^{N_u} X_{ui}   -X_u\bigg)  \mathbf{1}_{\{ {N_u}/{m_n}\leq  n^{-\lambda_\delta}\Pi_n \}}
\\  && \qquad  \qquad \qquad \qquad \qquad-      \E_{\xi,n}  \bigg ( \frac{1}{m_n} \sum_{i=1}^{N_u} X_{ui}   -X_u\bigg)  \mathbf{1}_{\{ {N_u}/{m_n}\leq  n^{-\lambda_\delta}\Pi_n \}}  \Bigg]^2\\
    &&\leq \frac{n^{2\lambda_\delta}}{\Pi_n^2} \E_\xi  \sum_{u\in\T_n}  \E_{\xi,n}\Bigg[\bigg ( \frac{1}{m_n} \sum_{i=1}^{N_u} X_{ui}   -X_u\bigg)  \mathbf{1}_{\{ {N_u}/{m_n}\leq  n^{-\lambda_\delta}\Pi_n\}} \Bigg]^2 \\
    &&=    \frac{n^{2\lambda_\delta}}{\Pi_n^2} \E_\xi  \sum_{u\in\T_n}  \E_{\xi}\Bigg[\bigg ( \frac{1}{m_n} \sum_{i=1}^{N_u} X_{ui}   -X_u\bigg) ^2 \mathbf{1}_{\{ {N_u}/{m_n}\leq  n^{-\lambda_\delta}\Pi_n \}}  \Bigg| \F_n\Bigg]
\\ &&\leq \frac{K_\xi n^{3+2\lambda_\delta}}{\Pi_{n} }  \E_\xi \frac{\widehat{N}_n^2}{m_n^2}\mathbf{1}_{\{ \widehat{N}_n/{m_n}\leq  n^{-\lambda_\delta}\Pi_n \}} +\frac{K_\xi n^{3+2\lambda_\delta}}{\Pi_{n} }
\\ &&=\frac{K_\xi n^{3+2\lambda_\delta}}{\Pi_{n} }  \E_\xi \frac{\widehat{N}_n^2}{m_n^2}\Big( \mathbf{1}_{\{  {\widehat{N}_n}  /m_n\leq  e^{2\lambda} \}}+ \mathbf{1}_{\{ e^{2\lambda}<  {\widehat{N}_n}/m_n  \leq  n^{-\lambda_\delta}\Pi_n \}} \Big) +\frac{K_\xi n^{3+2\lambda_\delta}}{\Pi_{n} }
\\ &&  \leq  \frac{K_\xi n^{3+2\lambda_\delta}}{\Pi_{n} }  \E_\xi \frac{\widehat{N}_n^2}{m_n^2} \mathbf{1}_{\{ e^{2\lambda}<  {\widehat{N}_n}/m_n  \leq  n^{-\lambda_\delta}\Pi_n \}}+\frac{K_\xi n^{3+2\lambda_\delta}}{\Pi_{n} }\\
&&\leq \frac{K_\xi n^{3+2\lambda_\delta}}{\Pi_{n} }\frac{\Pi_{n}  }{n^{\lambda_\delta}} \bigg(\ln \frac{\Pi_{n}  }{n^{\lambda_\delta}} \bigg)^{-1-\lambda} \E_\xi \frac{\widehat{N}_n^2}{m_n^2} \bigg[\frac{\widehat{N}_n}{m_n} \bigg(\ln^+ \frac{\widehat{N}_n}{ m_n}\bigg)^{-1-\lambda} \bigg]^{-1}+\frac{K_\xi n^{3+2\lambda_\delta}}{\Pi_{n} }
\\ &&
\quad (\mbox{ because }  x(\ln x)^{-1-\lambda} \mbox{ is increasing for }  x > e^{2\lambda})
\\ && \leq  K_\xi   n^{2+\lambda_\delta-\lambda}\bigg (\E_\xi \frac{\widehat{N}_n}{m_n}\Big(\ln^+\widehat{N}_n\Big)^{1+\lambda} +(\ln^- m_n)^{1+\lambda}  \bigg)+\frac{K_\xi n^{3+2\lambda_\delta}}{\Pi_{n} }.
\end{eqnarray*}
By the above estimates  and  \eqref{EEBeq4.5}, we see that   the series $\sum_{n=1}^\infty   n^{2\lambda_\delta} \E_\xi(\textsc{I}_n'- \E_{\xi,n}\textsc{I}_n'
 )^2$  converges a.s.

So we have proved the three series in \eqref{EEBeq4.2}  converges a.s. and hence \eqref{EEBeq4.1} holds. By setting $V_3=  \sum_{n=1}^\infty \textsc{I}_n + {N}_{3,1}$,
we have $N_{3,n}-V_3= \sum_{j=n}^\infty  \textsc{I}_j$, and hence part (3) of  the lemma follows from Lemma \ref{EEB-lem4} and \eqref{EEBeq4.1}.

\end{proof}

\section*{Acknowledgements}
We are grateful to the editor in charge of the article and to two anonymous referees for their very valuable comments, remarks or suggestions, which significantly contributed to improving the quality of the article. The work has been partially supported
by the National Natural Science Foundation of China (Grants No. 11101039, No. 11271045, No. 11571052, No. 11401590), by the Fundamental Research Funds for the Central Universities (China), and by the Natural Science Foundation of Guangdong Province (China), Grant no. 2015A030313628.
Q. Liu has also benefited from a delegation CNRS (France) during the preparation of the article.

\bibliographystyle{imsart-number}
\bibliography{Gao_20160216}

\end{document}